\documentclass[a4paper,11pt,reqno]{amsart}
\usepackage[T1]{fontenc}
\usepackage{amsmath}
\usepackage{tikz}
\usetikzlibrary{calc}
\usepackage[pagewise]{lineno}
\usepackage{latexsym,amsfonts,amssymb,amsthm,euscript, subfigure}
\usepackage{graphicx}
\usepackage{hyperref}
\usepackage{cleveref}
\usepackage{float}
\usepackage{color,fancybox}
\usepackage[normalem]{ulem}  %sublinhar texto
\usepackage{latexsym}
\usepackage{enumitem}
\usepackage{comment}
\usepackage[british]{babel}
\usepackage{indentfirst}

\usepackage{mathrsfs}

\newcounter{hipotese}

\theoremstyle{plain}% default
\newtheorem{Teo}{Theorem}[section]

\newtheorem{Ex}[Teo]{Example}
\newtheorem{Lema}[Teo]{Lemma}
\newtheorem{Prop}[Teo]{Proposition}
\newtheorem{remark}[Teo]{Remark}
\newtheorem{Cor}[Teo]{Corollary}
\newtheorem{maintheorem}{Theorem}

\theoremstyle{remark}

\newcommand{\diam}{\operatorname{diam}}

\newcommand{\Var}{\operatorname{Var}}
\newcommand{\n}{\widetilde{N}}
\newcommand{\nn}{\widetilde{n}}

\newcommand{\eps}{\varepsilon}
\newcommand{\de}{\delta}
\newcommand{\vi}{\vartheta}

\newcommand{\N}{\mathbb{N}}

\begin{document}

%\linenumbers
\title[Equilibrium states for non-uniformly hyperbolic systems]{Equilibrium states for non-uniformly hyperbolic systems: statistical properties and  analyticity}

\author[S. M. Afonso]{S. M. Afonso}
\address{Suzete Maria Afonso\\ Universidade Estadual Paulista (UNESP), Instituto de Geoci\^{e}ncias e Ci\^{e}ncias Exatas, C\^ampus de Rio Claro,  Avenida 24-A, 1515, Bela Vista, Rio Claro, S\~ao Paulo,\\ 13506-900, Brazil}
\email{s.afonso@unesp.br} 

\author[V. Ramos]{V. Ramos}
\address{Vanessa Ramos \\ Centro de Ci\^{e}ncias Exatas e Tecnologia-UFMA\\ Av. dos Portugueses, 1966, Bacanga\\  65080-805 S\~{a}o Lu\'{i}s\\Brazil}
\email{ramos.vanessa@ufma.br}

\author[J. Siqueira]{J. Siqueira}
\address{Jaqueline Siqueira\\ Departamento de Matem\'atica, Instituto de Matem\'atica, Universidade Federal do Rio de Janeiro, Caixa Postal 68530, Rio de Janeiro, RJ 21945-970, Brazil}
\email{jaqueline@im.ufrj.br}

\date{}
\thanks{The authors  were supported by grant 2017/08732-1, S\~ao
Paulo Research Foundation (FAPESP). VR was supported  by grant Universal-01309/17, FAPEMA-Brazil . JS was supported by grant Universal-430154/2018-6, CNPq-Brazil.}
\pagenumbering{arabic}

\begin{abstract} We consider a wide family of non-uniformly expanding maps and hyperbolic H\"older continuous potentials. We prove that the unique equilibrium state associated to each element of this family is given by the eigenfunction of the transfer operator and the eigenmeasure of the dual operator (both having the spectral radius as eigenvalue). We show that the transfer operator has the spectral gap property in some space of H\"older continuous observables and from this we obtain an exponential decay of correlations and a central limit theorem for the equilibrium state. Moreover, we establish the analyticity  with respect to the potential of the equilibrium state as well as that of other thermodynamic quantities. Furthermore, we derive similar results for the equilibrium state associated to a family of non-uniformly hyperbolic skew products and hyperbolic H\"older continuous potentials.
\end{abstract}

\maketitle

%\tableofcontents

%%%%%%%%%%%%%%%%%%%%%%%%%%%%%%%%%
%%%%%%%%%%%%%%%%%%%%%%%%%%%%%%%%%%%%%%%%%%%%%%%%%%%%%%%%%%%%%%%%%%%%%%%%%%%%%%%%%%%%%%%%%%%%%%%%%%%%%%%
%%%%%%%%%%%%%%%%%%%%%%%%%%%%%%%%%%%%%%%%%%%%%%%%%% Introducao %%%%%%%%%%%%%%%%%%%%%%%%%%%%%%%%%%%%%%%

\section{Introduction}
 The field of ergodic theory has been developed  with the goal of understanding the statistical behavior of a dynamical system via measures which remain invariant under its action.
In general, a dynamical system admits more than one invariant measure which makes it necessary  to choose a suitable one  to analyze the system. In order to do it, one may select  those maximizing the free energy of the system, called  {\em equilibrium states}.
 
Formally, given a continuous map $f: M \to M$ defined on a compact metric space $M$ and a  continuous potential $\phi: M \to \mathbb{R}$, we say that an \linebreak $f$-invariant probability measure $\mu$ %on the Borel sets of $M$ 
is an {\em equilibrium state} for $(f,\phi)$  if it satisfies the following variational principle:
          $$h_{\mu} (f) + \int \phi \, d\mu  = \sup_{\eta\in\mathbb P_f(M)} \left\{ h_{\eta} (f) + \int \phi \, d\eta 
\right\},  $$
 where $\mathbb P_f(M)$ denotes the set of $f$-invariant probability measures on the Borel sets of $M$ endowed with the weak* topology.

The study of equilibrium states was initiated by Sinai, Ruelle and Bowen in the seventies through the application of techniques and results from statistical physics to smooth dynamical systems. Sinai, in his pioneering work \cite{Sinai}, studied the problem of existence and finiteness of equilibrium state for Anosov diffeomorphisms and  H\"older continuous potentials. This strategy was carried out by Ruelle in \cite{Ruelle68}, \cite{Ruelle78} and Bowen in  \cite{Bowen71}  to extend the theory to uniformly hyperbolic (Axiom A) dynamical systems.
%%%% nao alteramos
%In the non-uniformly hyperbolic setting in dimension greater than one, several advances were obtained by Sarig (see \cite{Sarig99}, \cite{Sarig03}) and Buzzi   \cite{BuzziSarig}, who studied countable Markov shifts, and by Buzzi, Paccaut and Schmitt \cite{Buzzi},  who studied piecewise expanding maps.   Arbieto, Matheus and Oliveira \cite{AMO}, Oliveira and Viana \cite{OV08}, and Varandas and Viana \cite{VV} studied certain classes of non-uniformly expanding maps.
%%%%%%
Since then, the study of this problem has been greatly extended 
%the problem of existence and finiteness  of equilibrium states 
to include classes beyond uniform hyperbolic systems  among many others contributions we cite the recent works
%has been adressed for many authors who made valuable contributions to the development of the field, 
 \cite{BuzziSarig}, \cite{Clime2018}, \cite{Clime2016} and \cite{sarig_etds}.

Once  the existence and  finiteness of equilibrium states have been established,  it is natural to ask which type of information one can  obtain regarding the system via the equilibrium state. For instance, how fast is the memory of the past  lost as time evolves? In other words, is it possible to specify the rate of decay of correlations?  Also, can one characterize weak correlations via a central limit theorem? Moreover, what type of regularity can one obtain for the equilibrium state when the potential varies? 

In the context of  one dimensional piecewise expanding maps, Liverani \cite{Liverani2} proved  an exponential decay of correlations  and Keller \cite{Ke} obtained a central limit theorem.  
Such statistical properties  were  also obtained for the unique equilibrium state associated to 
 potentials with small variation for a class of  non-uniformly expanding maps by Castro and Varandas \cite{CV} and for a class of partially hyperbolic horseshoes by Ramos and Siqueira \cite{RS17}.  

Bruin and Todd \cite{Bruin-Todd2008} considered a class of smooth interval maps  and proved analyticity of the topological pressure for a one-parameter family of  potentials with bounded range: the variation is smaller than the topological entropy. For  smooth deformations of generic nonuniformly hyperbolic unimodal  maps, Baladi and Smania \cite{Baladi-Smania-2012} proved the differentiability of the absolutely continuous invariant measure.
For a family  of real  multimodal maps,  Iommi and Todd   \cite{Iommi-Todd} gave a characterization  of the first order phase transitions  of the topological pressure for the geometric potentials. 
In the context of countable Markov shifts, Sarig \cite{sarig-pressao} studied the analyticity of the topological pressure for some one-parameter family of potentials. 
 For a class  of non-uniformly expanding maps,  Bonfim, Castro, and Varandas \cite{BCV} obtained  linear response formulas for its equilibrium states. 
%using  projective metrics of Castro and Varandas was carried ou by Ramos and Siqueira \cite{RS17}.

%For a class of partially hyperbolic horseshoes and potentials with small variation,   Ramos and Siqueira \cite{RS17} established  exponential decay of correlations and a central limit theorem for its  unique equilibrium state.

%For a class of partially hyperbolic systems semiconjugated to nonuniformly expanding maps Castro and Nascimento \cite{CN} proved exponential decay of correlations and a central limit theorem for the maximal entropy measure. 
In this paper we study a wide class of non-uniformly hyperbolic maps associated to hyperbolic potentials with small variation. The uniqueness of equilibrium states for this class was established in \cite{RV16}; here we  derive that  such equilibrium state admits strong statistical properties. Namely, the correlations decay exponentially and a central limit theorem holds. We use the approach of projective metrics to prove that the transfer operator has the spectral gap property in a suitable space of H\"older continuous observables. 
 The core of the paper is to establish this spectral gap property. From this several nice properties can be obtained applying classical analytical methods.

This technique has been successfully implemented by several authors,  such as Hofbauer and Keller  \cite{Hofbauer1982} for  studying piecewise monotonic transformations,  Baladi \cite{Baladi}  and Liverani \cite{Liverani} for studying hyperbolic maps, and  Young \cite{Young} for a class of non-uniformly hyperbolic maps.   Melbourne and Nicol \cite{Melbourne-Nicol} developed an approach in an abstract setting  to obtain statistical properties via the quasi-compactness of the transfer operator. In addition, Giulietti, Kloeckner, Lopes and Marcon \cite{Benoit} used  a differential geometric approach to obtain consequen\-ces of the spectral gap property.  
In this paper  we  follow the  approach of  \cite{RS17}, where a class of partially hyperbolic horsehoes were considered. 

In the context of this work,  from the spectral gap property, besides deriving statistical properties, it allows us to study the behavior of the system under small perturbations of the potential. Namely, we prove that the equilibrium state, as well as  others thermodynamical quantities, such as the topological pressure, vary analytically.
In \cite{ARS},  it was proved that the equilibrium state is  {\em jointly} continuous with respect to the map and the potential while here we establish {\em analyticity} but only as a function of the {\em potential}. 

The main idea for establishing the spectral gap property in the space of H\"older continuous observables is  proving that the transfer operator contracts a suitable cone. This requires a contractive behavior in the pre-images of points that are close enough. The strong topological mixing property guarantees that each point has at least one pre-image in the non-uniformly expanding  set, giving us a good control of these pre-images. The other pre-images are then controlled by a domination condition.

This paper is organized as follows. In Section~\ref{Definitions}  we describe the setting and state the main results. %Moreover, we define the transfer operator and its spectrum.
In Section~\ref{preliminaries} we introduce some definitions and results used throughout the paper. %, including the concept of  topological pressure relative to a (not necessarily compact) set,  hyperbolic times, properties of projective metrics of convex cones, and some  classical results from the   spectral theory of bounded operators. 

In Section~\ref{cones} we  prove that the transfer operator admits the  spectral gap property. In Section~\ref{Statistical} we use the spectral gap property to obtain
 statistical properties for the equilibrium state.  We also describe the thermodynamical formalism, that is, we prove that the unique equilibrium state is given by   an eigenfunction of the transfer operator and an eigenmeasure of its dual both having the spectral radius as an eigenvalue.
In Section~\ref{anal} we prove  that the equilibrium state, as well as  other thermodynamical quantities vary analytically with the potential. 

In Section \ref{skew product}	we extend our results to a wider class of  non-uniformly hyperbolic maps. Finally, in Section~\ref{exemplos} we present several examples for which our results hold.

%%%%%%%%%%%%%%%%%%%%%%%%%%%%%%%%%%%%%%%%%%%%%%%%%%%%%%%%%%%%%
%%%%%%%%%%%%%%%%%%%%%%%%%%%%%%%%%%%%%%%%%%%%%%%%%%%%%%%%%%%%%%%%%%%%%%%%%%%%%%%%%%%%%%%%%%%%%%%%%%%%%%%%%%%%%%%%%%%%%%%%%%%%%%%%%%%%%%%%%%%%         Main                    Results                         %%%%%%%%%%%%%%%%%%%%%
%%%%%%%%%%%%%%%%%%%%%%%%%%%%%%%%%%%%%%%%%%%%%%%%%%%%%%%%%%%%%%
%%%%%%%%%%%%%%%%%%%%%%%%%%%%%%%%%%%%%%%%%%%%%%%%%%%%%%%%%%%%%%

\section{Definitions and main results}\label{Definitions}
%\textcolor{purple}{A verificar }
%
%\textcolor{purple}{ 1. a escolha de theta}
%
%\textcolor{purple}{Podemos incluir uma nota apenas dizendo que theta eh estavel sob pequenas perturbacoes c1 ou notar que theta depende de f }
%
%\textcolor{purple}{2. Para provar analiticidade usamos que o teo 3.11 para ver que eh infinitamente diferenciavel( com a nocao de diferenciabilidade de Frechet) e entao aplicamos o Teo 3.10 para garantir que infinitamente dif implica analitico.}
%\textcolor{purple}{ 2.1. correto??? 2.2  o texto da secao 3.4 me parece agora bom, concordam??}

Let $M$ be a compact Riemannian manifold and  let $\mathcal F$ be a family of $C^1$ local diffeomorphisms $f:M\to M$ . 
Given $\sigma\in(0, 1)$, define $\Sigma_{\sigma}(f)$ as the set of points $x\in M$ where $f$ is \emph{non-uniformly expanding}, i.e. 
\begin{eqnarray} \label{limsup}
\limsup_{n\rightarrow+\infty}\frac{1}{n}\sum_{i=0}^{n-1}\log\|Df(f^{j}(x))^{-1}\|\leq \log\sigma.
\end{eqnarray}
We say that a continuous potential $\phi: M \to \mathbb{R}$  is  {\em $\sigma$-hyperbolic  for~$f$} if  the  topological pressure of $\phi$ (with respect to $f$) is equal to the relative pressure of $\phi$ on the set~$\Sigma_{\sigma}(f)$; we recall the definition of topological pressure relative to a set in  Section~\ref{preliminaries}, Subsection \ref{tp}.

Given  $\alpha>0$, consider $  C^\alpha(M)$ the space of H\"older continuous functions $\varphi:M\to\mathbb R$ endowed with the seminorm
 $$|\varphi|_\alpha=\sup_{x\neq y}\frac{|\varphi(x)-\varphi(y)|}{d(x,y)^\alpha}
 $$
and the norm
 $$\|\varphi\|_\alpha=\|\varphi\|_0+|\varphi|_\alpha,$$
 where $\|\quad\|_0$ stands for the $\sup$ norm in $C^0(M)$. We shall always consider  $ \mathcal F\times  C^\alpha(M)$ endowed  with the product topology. 
 
Given  $(f, \phi) $, with  $\phi$ a $\sigma$-hyperbolic  potential, under the  assumption that
\begin{equation*}\label{predensa}
\{f^{-n}(x)\}_{n\ge 0} \ \text{is dense in}\ M \ \text{for all} \ x\in M,
\end{equation*}
the existence of a unique equilibrium state $\mu_{f,\phi}$ was established in [Theorem 2,\cite{RV16}]. Clearly, the condition above holds whenever $f$ is \emph{strongly topologically mixing}, i.e. 
\begin{equation*}%\tag{$\ast$}
\text{ for every open set } U\subset M \text{  there is } N\in\mathbb N \text{ such that } f^N(U)=M.
\end{equation*}

Consider $\mathcal{Q}$ a cover by injectivity domains of $f$. Since $M$ is compact and $f$ is a local diffeomorphism, we can take $\mathcal{Q} =\{U_1, \dots, U_m\}$ for some $m\in \N$. 
Defining
\begin{equation*}\label{l}
\vartheta = \max_{1\leq j\leq m}\left\{\sup_{x\in U_j} \|Df^{-1}(x)\|\right\},
\end{equation*}
we note that $\vartheta$ does not depend on the choice of the cover $\mathcal{Q}$. 

We fix $\sigma\in(0, 1)$ satisfying
\begin{equation}\tag{$\ast$} \label{h1}
\vartheta \cdot \sigma <1. 
\end{equation}

Consider the family $\mathcal H_{\sigma}$, consisting
of pairs $(f,\phi) \in \mathcal F\times C^\alpha(M) $ such that $f$ is strongly topologically mixing, satisfies condition \eqref{h1} and  $\phi$ is $\sigma$-hyperbolic for $f$ satisfying condition \eqref{2 estrelas} that will be properly stated in Section~\ref{cones}.

We establish  that for each $(f,\phi)\in \mathcal H_{\sigma}$  its unique equilibrium state $\mu_{f,\phi}$ has an exponential decay of correlations for H\"older continuous observables.

\begin{maintheorem}\label{decaimento f}
	There exists  a constant $0<\tau < 1$ such that for all $\varphi \in L^{1}(\mu_{f,\phi}), \psi \in C^{\alpha}(M)$ there exists $K:=K(\varphi, \psi)>0$ satisfying 
	$$\left|  \int \left(\varphi \circ f^n\right)\cdot\psi \, d\mu_{f,\phi} - \int \varphi \, d\mu_{f,\phi}\int\psi \, d\mu_{f,\phi} \right|   \leq \ K \cdot \tau^{n}  \quad \mbox{for every} \ n\geq 1. $$
\end{maintheorem}

We also derive a central limit theorem for the equilibrium state $\mu_{f,\phi}$. 

\begin{maintheorem}\label{TCL}
	Let $\varphi $ be an $\alpha$-H\"older continuous function and let $\tilde{\sigma} \geq 0$ be defined by 
	$$\tilde{\sigma}^{2}= \int \psi^2  \ d\mu_{f,\phi} + 2\displaystyle\sum_{n=1}^{\infty}  \int \psi (\psi \circ f^n) \ d\mu_{f,\phi}, \quad \mbox{where} \quad \psi = \varphi - \int \varphi \ d\mu_{f,\phi} .$$
	Then $\tilde{\sigma}$ is finite and $\tilde{\sigma} = 0$ if and only if $\varphi = u \circ f - u $ for some $u \in L^{2}(\mu_{f,\phi})$. On the other hand, if $\tilde{\sigma} >0 $ then given any interval $A\subset \mathbb{R}$,
	$$\mu_{f, \phi}\left(x\in M : \frac{1}{\sqrt{n}} \sum_{j=0}^{n} \left( \varphi(f^j(x))   - \int \varphi \ d\mu_{f,\phi} \right) \in A \right) \to \frac{1}{\tilde{\sigma} \sqrt{2\pi}} \int_{A} e^{-\frac{t^2}{2\tilde{\sigma}^2}} \ dt,$$ 
	as $n$ goes to infinity.
\end{maintheorem}

%%%%%%%%%%%%%%%%%%%%%%%%%%%%%%%%%%%%%%%%%%%%%%%%%%%%%%%%%%%%%%%%%%
%%%%%%%%%%%%%%%%%%%%%%%%%%%% Transfer Operator %%%%%%%%%%%%%%%%%%%%%%%%%%%
 %%%%%%%%%%%%%%%%%%%%%%%%%%%%%%%%%%%%%%%%%%%%%%%%%%%%%%%%%%%%%%%%%%

\subsection*{Transfer operator  and its spectrum }
\label{operador transf}

Given $\sigma\in(0,1)$ let  $(f, \phi) \in \mathcal{H}_{\sigma}$.
Denote by $C^{0}(M)$ the set of real continuous functions on $M$ endowed with the sup norm. 

We define the operator $\mathcal{L}_{f,\phi} :C^{0}(M)\rightarrow C^{0}(M)$  called the {\emph{Ruelle--Perron--Frobenius operator}} or simply the {\emph{transfer operator}}, which associates to each $\psi \in C^{0}(M)$ a continuous function
$ \mathcal{L}_{f,\phi} (\psi) \colon M \to \mathbb{R}$ by
$$\label{optransf}
\mathcal{L}_{f,\phi} \psi \left(x\right) = \displaystyle\sum _{y  \in \, f^{-1}\left(x\right)} e^{\phi(y)} \psi \left( y \right). 
$$

Note that $\mathcal{L}_{f,\phi}$ is a positive and bounded linear operator. 
Consider the resolvent set $$Res({\mathcal{L}_{f,\phi}}):= \{ z\in \mathbb{C} | (zI - {\mathcal{L}_{f,\phi}} ) \ \mbox{has a bounded inverse}  \},$$
and its complement set, called the spectrum set,
$$Spec ({\mathcal{L}_{f,\phi}}):= \{ z\in \mathbb{C} : (z I - \mathcal{L}_{f,\phi}) \ \mbox{has no bounded inverse}  \}.$$

{We say that the transfer operator $\mathcal{L}_{f,\phi}$ satisfies the \emph{spectral gap property} if its spectrum set $Spec(\mathcal{L}_{f,\phi})\subset \mathbb{C} $ admits a decomposition as follows: $Spec(\mathcal{L}_{f,\phi}) =\left\{ \lambda \right\} \cup \Sigma$ where $\lambda\in \mathbb{R}$ is an eigenvalue for $\mathcal{L}_{f,\phi}$ associated to a one-dimensional eigenspace and $ \Sigma$ is strictly contained in a ball centered at zero and of radius strictly less than $\lambda.$}

It should be mentioned that besides of the {\it spectral gap property}   having its intrinsic interest,  this   property is the core of this paper since it will  be used  to obtain our main results.

\begin{maintheorem}\label{gap spectral}
For every $(f,\phi) \in \mathcal{H}_{\sigma}$ the transfer operator, $\mathcal{L}_{f,\phi}$, has the spectral gap property in the space of H\"older continuous observables.  
\end{maintheorem}

%
%For each $n \! \in\! \mathbb{N}$, denoting by $S_{n}\phi$ the Birkhoff sum $S_{n}\phi(x)= \displaystyle\sum_{j=0}^{n-1} \phi\big(f^{j}(x)\big)$, we have that 
%$$\label{iteradostransf}
%\mathcal{L}_{f,\phi}^{n} \psi \left(x\right) = \displaystyle\sum _{y  \in \, f^{-n}\left(x\right)} e^{S_{n}\phi \left(y \right)} \psi \left( y \right).
%$$

According to Riesz--Markov Theorem we can consider the dual operator $ \mathcal{L}_{f,\phi}^{\ast}:\mathbb P(M) \to \mathbb P(M)$ as the operator that   satisfies
$$\int \psi \ d\mathcal{L}_{f,\phi}^{\ast}\eta  = \int  \mathcal{L}_{f,\phi}(  \psi ) \ d\eta , $$
for every $\psi \in  C^{0}(M) $ and every $\eta \in \mathbb P(M)  $, where $ \mathbb P(M)$ denotes the space of probability measures on the Borel sets of $M$. 

Now consider the spectral radius $\lambda_{f,\phi}$ of the transfer operator defined by $\lambda_{f,\phi}:= \sup \{ |z| : z \in Spec(\mathcal{L}_{f,\phi})  \}$. Since $\mathcal{L}_{f,\phi}$ is positive, the spectral radius can be computed by the following formula
    $$\lambda_{f,\phi}=\displaystyle\lim_{n \to \infty} \displaystyle\sqrt[n]{\parallel {\mathcal{L}}_{f,\phi}^{n} \parallel }. $$
Using that $\|\mathcal{L}_{f,\phi}^{n} \|=\|\mathcal{L}_{f,\phi}^{n} 1 \|$ for any $n\geq 1$  we deduce that % the spectral radius $\lambda_{f,\phi}$ of the transfer operator satisfies
\begin{equation}\label{eq.radius}
\deg(f)e^{\inf \phi}\leq  
\lambda_{f,\phi}
\leq\textbf{\rm{deg}}(f)e^{\sup\phi}.
\end{equation}

As a byproduct of the spectral gap property  of the transfer operator we can describe the Thermodynamical Formalism for a pair $(f,\phi) \in \mathcal{H}_{\sigma}.$

\begin{maintheorem} \label{formalismo} Let $\lambda_{f,\phi}$ be the spectral radius of the transfer operator $\mathcal{L}_{f,\phi}$. There exist a probability measure $\nu_{f,\phi}\in\mathbb{P}(M)$ and a H\"older continuous function $h_{f, \phi}:M\to\mathbb{R}$ bounded away from zero and infinity which satisfies  
$$\mathcal{L}_{f,\phi}^{\ast}\nu_{f,\phi}=\lambda_{f,\phi}\nu_{f,\phi} \quad and \quad \mathcal{L}_{f,\phi}h_{f,\phi}=\lambda_{f,\phi} h_{f,\phi}.$$
Moreover, the invariant measure $\mu_{f, \phi}$ given by $\mu_{f, \phi}=h_{f, \phi}\nu_{f, \phi}$ is the unique equilibrium state associated to $(f, \phi)\in \mathcal{H}_{\sigma}$.
\end{maintheorem}

We also prove that some thermodynamical quantities vary analytically with the potential, in particular,  the equilibrium state as well.  We describe the setting as follows: consider $f\in \mathcal{F}$ strongly topologically mixing satisfying condition \eqref{h1}. Denote by $\mathcal{P}_\sigma$ the set of $\alpha$-H\"older continuous potentials which are $\sigma$-hyperbolic for $f$ and satisfy condition~\eqref{2 estrelas}.

\begin{maintheorem}\label{analiticidade potencial}
%Consider $F: \Omega \to \Omega$ a partially hyperbolic horseshoe as above and $f: \Lambda \to \Lambda$ its projected map. 
 %Given $f \in \mathcal{F}$ let $\mathcal{P}_\sigma$  be the set of $C^r$ differentiable and $\sigma$-hyperbolic potentials for $f$.
 %Then
  The following functions defined in $\mathcal{P}_\sigma$ are analytic:
\begin{enumerate} [label={\textup{(\roman*)}},itemsep=0pt]
%\item The transfer operator function $\phi\longmapsto \mathcal{L}_{f, \phi}\in L(C^{r+ \alpha}(\Lambda, \mathbb{R}))$;
\item The topological pressure function $\phi\longmapsto P_{f}( \phi)\in\mathbb{R}$;\label{i}
\item The invariant density function $\phi\longmapsto h_{f,\phi}\in C^{\alpha}(M)$;\label{ii}
\item The conformal measure function $\phi\longmapsto \nu_{f,\phi}\in (C^{\alpha}(M))^*$.\label{iii} 
\item The equilibrium state function $\phi\longmapsto \mu_{f,\phi}=h_{f, \phi}\nu_{f, \phi}\in (C^{\alpha}(M))^*$. \label{iv}
\end{enumerate}
\end{maintheorem}

The meaning of analyticity in this context is the one given in Section~\ref{anal}. 
We point out that according \cite{ARS} the family of hyperbolic potentials is an open class in the $C^0$-topology. Since \eqref{2 estrelas} is an open condition on the potential we derive that $\mathcal{P}_\sigma$ is an open subset of $C^{\alpha}(M).$

In Section~\ref{exemplos} we present some examples for which our results hold.

%\begin{maintheorem}\label{diferenciabilidade dinamica}
%Let $\phi:M\to\mathbb{R}$ be a potential belongs to  $\mathcal{E}^{r}_{\ast}$. The topological pressure function $\mathcal{P}\ni f\mapsto P_{\phi}(f)=P_{top}(f, \phi)\in \mathbb{R}$ is $C^1$ differentiable.
%\end{maintheorem}
%
%

%%%%%%%%%%%%%%%%%%%%%%%%%%%%%%%%%%%%%%%%%%%%%%%%%%%%%%%%%%%%%%%%%%%%%%%%%%%%%%%%%%%%%%%%%%%%%%%%%%%%%%%%%%%%%%%%%%%%%%%%%%%%%%%
%%%%%%%%%%%%%%%%%%%%%%%%%%%%     APPLICATIONS   %%%%%%%%%%%%%%%%%%%%%%%%%%%%%%%%%%%%%%%%%%%%%%%%%%%%%%%%%%%%%%%%%%%%%%%%%%%%%%%%%%%%%%%%%%%%%%%%%%%%%%%%%%%%%%%%%%%%%%%%%%%%%%%%%%%%%%%%%%%%%%%%%%%%%%%%
%%%%%%%%%%%%%%%%%%%%%%%%%%%%%%%%%%%%%%%%%%%%%%%%%%%%%%%%%%%%%%%%%%%%%%%%%%%%%%%%%%%%%%%%%%%%     PRELIMINARES    %%%%%%%%%%%%%%%%%%%%%%%%%%%%  %%%%%%%%%%%%%%%%%%%%%%                                    %%%%%%%%%%%%%%%%%%%%%%%%%%%%%%%%%%%%%%%%%%%%%%%%%%%%%%%%%%%%%%%%%%%%%%%%%%%%%%%%%%%%%%%%%%%%%%%%%%%%%%%%%%%%%%%%%%%%%%%%%%%%%%%%%%%%%%%%%%%%%%%%%%%%%%%%%

\section{Preliminaries} \label{preliminaries}

In this section we present some basic definitions and results that will be useful in the following sections. We begin with the definition  of pressure relative  to a set, not necessarily compact.  This will allow us to precise the concept of hyperbolic potentials stated in the last section.
%This approach comes from the dimension theory and it is very useful to calculate the topological pressure of non-compact sets. 

\subsection{Topological pressure} \label{tp}Let $M$ be a compact metric space and consider $T:M\to M$  and  $\phi:M\to\mathbb{R}$ both continuous.
Given $\delta>0$, $n\in\mathbb{N}$ and $x\in M$, define the \emph{dynamical ball} by  %,  $B_{\delta}(x,n)$, by the set 
$$B_{\delta}(x,n)=\left\{y\in M:\;d(T^{j}(x), T^{j}(y))<\delta,\;\mbox{for}\;\; 0\leq j\leq n\right\}.$$
For each  $N\in\mathbb N$ consider $\mathcal{F}_{N}$ the following collection of dynamical  balls 
 $$\mathcal{F}_{N}=\{B_{\delta}(x,n):\; x\in M\;\mbox{and}\; n\geq N\}.$$
Given $\Lambda\subset M$, not necessarily compact, denote by $\mathcal{F}_{N}(\Lambda)$ the finite or countable families of elements in $\mathcal{F}_{N}$ which cover $\Lambda.$
 For $n\in\mathbb N$, let %$S_n\phi$ be the Birkhoff sum
 %$$S_n\phi(x)=\phi(x)+\phi(T(x))+\cdots+\phi(T^{n-1}(x)).
 %$$
 %and
 $$R_{n,\delta}\phi(x)=\sup_{y\in B_\delta(x,n)}S_n\phi(y),$$
where $S_n\phi(y):= \sum_{j=0}^{n-1}\phi(f^j(y))$ is the Birkhoff's sum.  

Assume $\Lambda\subset M$ is invariant under $T$  and define, for each~$\gamma>0$,
$$
m_{T}(\phi, \Lambda, \delta, N,\gamma)=
\inf_{\mathcal{U}\subset\mathcal{F}_{N}(\Lambda)} \left\{\sum_{B_{\delta}(x,n)\in\mathcal{U}}e^{-\gamma n+ R_{n,\delta}\phi(x)}\right\}.
$$
Define
$$m_{T}(\phi, \Lambda, \delta, \gamma)=\lim_{N\rightarrow +\infty}{m_{T}(\phi, \Lambda, \delta, N,\gamma)},$$
and %by taking the infimum over $\gamma$ we obtain
$$P_{T}(\phi, \Lambda, \delta)=\inf{\{\gamma >0 \, | \; m_{T}(\phi, \Lambda, \delta,\gamma)=0\}}.$$
The \textit{relative pressure} of $\phi$ %$P_{T}(\phi, \Lambda)$ of a subset 
on $\Lambda$ is  defined by%of $X$ is defined by
 $$P_{T}(\phi, \Lambda)=\lim_{\delta\rightarrow 0} {P_{T}(\phi, \Lambda, \delta)}.$$
The \emph{topological pressure of $\phi$} is by definition $P_T(\phi,M)$, and it satisfies
\begin{equation}\label{pressure}
P_{T}(\phi)=\sup\left\{ P_{T}(\phi, \Lambda),\, P_{f}(\phi, \Lambda^{c})\right\},
\end{equation}
where $\Lambda^c$ is the complement of the set $\Lambda$ on $M$. We refer the reader to [Section 11,\cite{Pes}] for the proof of~\eqref{pressure} and for additional  properties of the pressure.

We precise the definition of hyperbolic potentials, first we consider potentials with respect to diffeomorphisms and after with respect to skew products. 

Let $M$ be a compact Riemannian manifold and let $f:M\to M$ be a local $C^1$ diffeomorphism.  Considering the set $\Sigma_{\sigma}(f)$ as in Section~\ref{Definitions}, we point out that it is an invariant set although not necessarily compact.   A real continuous function $\phi:M\to\mathbb{R}$ is said to be a {\em hyperbolic potential} for $f$ if % the topological pressure of the system $(f, \phi)$ satisfies
         $$P_f(\phi, \left(\Sigma_{\sigma}(f)\right)^{c})<P_f(\phi, \Sigma_{\sigma}(f))=P_f(\phi).$$
The class of hyperbolic potentials for a map consists in an open class in the $C^{0}$-topology [\cite{ARS}].

%where $\Sigma_{\sigma}(f)$.
%In the next result we  show that for every H\"older continuous potential which is hyperbolic for $F$ its possible to construct a potential homologous to it which does not depend on the stable direction and remains hyperbolic for $F$. 
%We say that two potentials  $\bar{\phi},\phi: M\times N\to \mathbb R$ are  \emph{homologous} if there is a continuous function $u:M \times N \to\mathbb{R}$ such that $\bar{\phi}=\phi-u+u\circ F$;

%%%%%%%%%%%%%%%%%%%%%%%%%%%%%%%%%%%%%%%%%%%%%%%%%%%%%%%%%%%%%%%%%%%%%%%%%%%%%%%%%%%%%%%%%%%%% HYPERBOLIC PRE BALLS         % % %%%%%% %%%%%%%%%%%%%%%%%%%%%%%%%%%%%%%%%%%%%%%%%%%%%%%%%%%%%%%%%%%%%%%%%%%%%%%%%%%%%%%%%%%%%%%%%%%%%%%%%%%%%%%%%%%%%%%%%%%%%%%%%%%%%%%%%%%%%%

\subsection{Hyperbolic pre-balls} \label{pre balls}
Here we present the concept of hyperbolic times that  will help us to deal with the lack of hyperbolicity in the family of maps that we consider. 
We say that $n$ is a \textit{hyperbolic time} for $x$ if 
$$\prod_{j=n-k}^{n-1}\|Df(f^{j}(x))^{-1}\|\leq \sigma^{ k/2}\quad\text{for all }1\leq k< n.$$
%\end{Def}
Since we consider only maps with no critical or singular sets, the definition of hyperbolic times given in \cite[Definition~5.1]{ABV} reduces  to the one that we have presented.

Condition \eqref{limsup} of non-uniform expansion  is enough to guarantee the existence of infinitely many hyperbolic times for points in $\Sigma_{\sigma}(f)$. See [Lemma 5.4, \cite{ABV}].

The next result guarantees that points associated to a hyperbolic time admit a neighborhood for which all  orbits behave as uniformly expanded ones.  We refer the reader to [Lemma 5.2, \cite{ABV}] for its proof.

\begin{Prop}\label{delta}
There exists $\delta>0$ such that  if  $n$ is a hyperbolic  time for $x\in M$, then there exists a neighborhood of $x$, $V_n(x)$, satisfying 
\begin{enumerate}[label={\textup{(\roman*)}},itemsep=0pt]
\item $f^n$ maps $V_n(x)$ diffeomorphically onto the ball centered on $f^n(x)$  and of radius $\delta$;
\item  for all $1\leq k< n$ and $y,x \in V_n(x)$,
  $$ d(f^{n-k}(y), f^{n-k}(z)) \leq \sigma^{k/2}d(f^{n}(y), f^{n}(z)). $$
\end{enumerate}
Moreover, for every $y\in V_n(x)$ we have $||Df^n (y)^{-1} || \leq \sigma^{n/2}$.
\end{Prop}

We will refer to the sets $V_n$ as {\em hyperbolic pre-balls}. Note that the $n$-th iterate of a hyperbolic pre-ball, $f^n (V_n)$, is actually a topological  ball of radius $\delta> 0$.
\begin{figure}[!htb]
	\centering
	\includegraphics[scale=0.8]{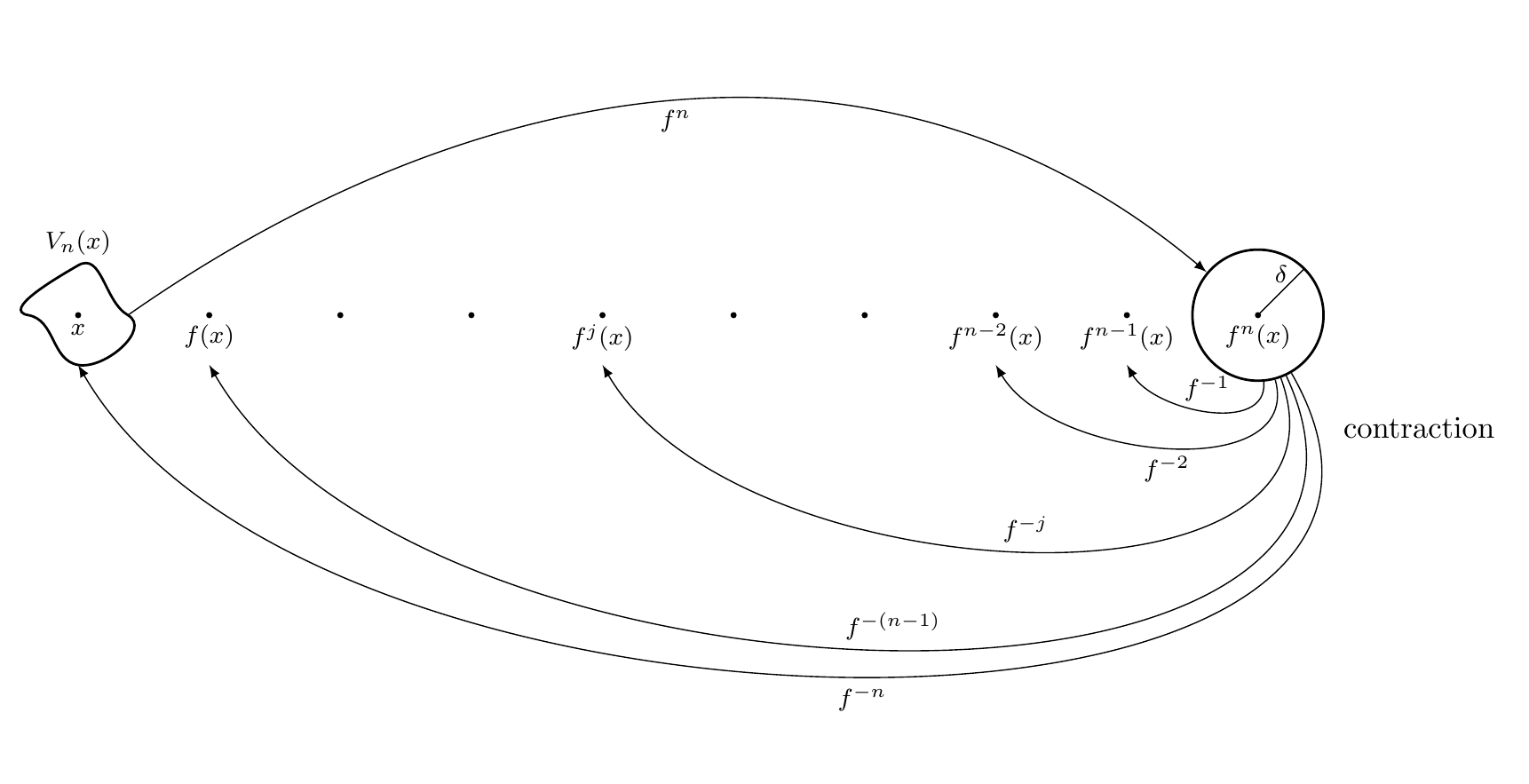}
	\caption{Hyperbolic pre-balls.}
\end{figure}

Given $f \in \mathcal{F}$,  fix $\delta>0$ provided by  Proposition \ref{delta}. Note that, $\delta>0$ depends only on $f$ and $\sigma$. %which with we can associate hyperbolic pre-balls. 
Moreover, $\delta$ can be taken uniformly in a neighborhood of $f$, see [Remark 3.5, \cite{ARS}].
From now on  we fix $\delta>0$  as above.

% take $\delta >0$  given by Proposition~\ref{delta}. Note that $\delta$ can be chosen   uniformly for a $C^1$ neighborhood of $f$.  The next result states that any topological ball of radius $\delta$ cover the manifold after some iterate of  the map $f$. 

\begin{Lema}\label{lemmahyper}
There exists $\widetilde{N}$ such that 
	\[
	f^{\widetilde{N}}(B(x,\delta)) = M \quad  \mbox{for all} \ x\in M.
	\]
Moreover,  for any $n\in \N$ and any hyperbolic pre-ball $V_n$ we have
	\[
	f^{\n + n} (V_n) = M
	\]
\end{Lema}	

\begin{proof}
	Consider a cover of $M$ by balls of radius $0<\eps<\frac{\de}{4}$. Since $M$ is a compact manifold, we can extract  a finite subcover of $M$,  namely \linebreak $\mathcal{B} =\{B_1, \dots, B_l\}$. From the property of topologically mixing of $f$, we can consider $N_1, \dots, N_l \in \N$ such that 
	\[
	f^{N_1}(B_1) = M,  \dots,  	f^{N_l}(B_l) = M.
	\]
	
	Let $\n = N_1 + \dots + N_l$. It is straightforward to see that
	\[
	f^{\n}(B_j) = M, \qquad 1\leq j\leq l.
	\]
	Given a ball $B(x,\delta)$ for some $x\in M$, there exists $B_{\tilde{l}}\in \mathcal{B}$, $1\leq \tilde{l}\leq k$, such that $B_{\tilde{l}}\subset B(x,\delta)$. 
	
	Therefore,
	\[
	f^{\n}(B(x,\delta)) = M. 
	\]
%\end{proof}

%\begin{Lema}\label{lemma2}
%	There exists $\n \in \N$ such that 
%	\[
%	f^{\n + n} (V_n) = M
%	\]
%	for any hyperbolic pre-ball $V_n$ and $n\in \N$. 
%\end{Lema}
%\begin{proof}
Now, for the second part of the lemma, consider the hyperbolic pre-ball $V_n(x)$ associated to some $x\in M$. From Proposition \ref{delta} we have 
	\begin{equation}\label{abv}
	f^n(V_n) = B(f^n(x), \de).
	\end{equation}
	On the other hand, from what we have proved %Lemma \ref{lemmahiper} there exists $\n\in \N$ such that 
	\begin{equation}\label{lant}
	f^{\n}(B(f^n(x), \de)) = M.
	\end{equation}
	Thus, by \eqref{abv} and \eqref{lant} we conclude that
	\[
	f^{\n + n}(V_n) = M. 
	\]
\end{proof}

Fix $\n \in \N$ given by the first part of the lemma. Consider
\[
\widetilde{n} = \inf_{n\in \N}\{n \,\, \mbox{hyperbolic time for some}\,\, x\in M |\, n\geq 2\n\}.
\]

Let $V_{\nn}(z)$ be the hyperbolic pre-ball associated to some $z\in M$. We define
\begin{equation}\label{ene}
N:= \n + \nn
\end{equation}
Note that $N > 3\n$.

By the second part of Proposition \ref{delta}, given $x,y\in M$, there exist $x_N,y_N\in V_{\nn}(z)$ such that 
\[
f^N(x_N) = x \qquad \mbox{and}\qquad f^N(y_N) = y.
\]
Moreover, if $d(x,y)<\de$, thus 
\begin{align}\label{hyper}
d(x_N, y_N) &\leq \sigma^{\nn/2} \cdot d(f^{\nn}(x_N), f^{\nn}(y_N))\nonumber\\
&\leq \sigma^{\nn/2} \cdot \vartheta^{\n} \cdot d(x,y)\nonumber\\
& \leq \sigma^{\nn/2 - \n}\cdot (\sigma\cdot \vartheta)^{\n} \cdot d(x,y)\nonumber\\
& = \gamma \cdot d(x,y),
\end{align}
 where we denote $\gamma = \sigma^{\nn/2 - \n}\cdot (\sigma\cdot \vartheta)^{\n}$. Note that $\gamma \ll 1$ since $\sigma <1$, $\nn>2\n$ and by hypothesis \eqref{h1} we have $\sigma\cdot \vartheta <1$.

 %%%%%%%%%%%%%%%%%%%%%%%%%%%%%%%%%%%%%%%%%%%%%%%%%%%%%%%%%%%%%%%%%%%%%%%%%%%%%%%%%%%%%%%%%%%%%%%%%%%%%%%%%%%%%%%%%%%%%%%%
%%%%%%%%%%%%%%%%%%%%%% METRICAS PROJETIVAS   %%%%%%%%%%%%%%%%%%%%%%%%
 %%%%%%%%%%%%%%%%%%%%%%%%%%%%%%%%%%%%%%%%%%%%%%%%%%%%%%%%%%%%
%%%%%%%%%%%%%%%%%%%%%%%%%%%%%%%%%%%%%%%%%%%%%%%%%%%%%%%%%%%

\subsection{Projective metrics}

%The next result establishes  that every locally H\"older continuous function defined on $M$ is H\"older continuous.

In this section we will state some definitions and results regarding projective metrics associated to convex  cones. 
The notion of projective metric associated to a convex cone in a vector space was introduced by Garrett  Birkhoff \cite{Birkhoff} and provides a nice way to explicit spectral properties of the transfer operator (see \cite{Baladi}, \cite{Liverani} and \cite{VianaStoch}, for instance).  In particular, we will derive the spectral gap property for the transfer operator through this notion.

Let $E$ be a Banach space. A subset $\mathcal{C}$ of $E  \!-\! \{0\}$ is called a \emph{cone} in $E$ if $C\cap (-C) = \{0\}$ and it satisfies:
\[
v\in C \quad \mbox{and} \quad \lambda >0 \quad \Rightarrow  \quad  \lambda \cdot v\in C. 
\]

A cone $C$ is called \emph{convex} if 
\[
v,w\in C \quad \mbox{and} \quad \lambda,\eta >0 \quad \Rightarrow  \quad  \lambda \cdot v + \eta\cdot w\in C. 
\]

The closure of $C$, denoted by $\bar{C}$, is defined by 
\[
\bar{C} := \left\{ w\in E| \,\, \mbox{there are}\,\, v\in C\,\, \mbox{and}\,\, \lambda_n\to 0 \,\,\mbox{such that}\,\,\right.\] \[\left. (w+\lambda_nv)\in C\,\, \mbox{for all}\,\, n\geq 1\right\}.
\]

A cone $\mathcal{C} $ is called \emph{closed} if $\bar{\mathcal{C}}= \mathcal{C} \cup \{0\}$.

Let $\mathcal{C}$ be a closed convex cone and given $v, w \in \mathcal{C} $ define 
\begin{equation*}\label{A e B}
A(v,w)= \sup \left\{t>0 : w-tv \in \mathcal{C}  \right\} \ \mbox{and} \ B(v,w)=\inf \left\{s>0 : sv -w \in \mathcal{C}  \right\},
\end{equation*}
with the convention $\sup \emptyset = 0$ and $\inf \emptyset = +\infty$, where $\emptyset$ denotes the empty set.

It is easy to verify that $A(v,w)$ is finite, $B(v,w)$ is positive and $A(v,w)\leq B(v,w)$ for all $v, w \in \mathcal{C}$ (see \cite{VianaStoch}). 
We set  
$$ \Theta(v,w) = \log\left( \frac{B(v,w)}{A(v,w)} \right), $$ 
with $\Theta$ possibly infinity in the case $A=0$ or $B=+\infty$.

By virtue of properties of $A$ and $B$, we derive that
 $\Theta(v,w)$ is well-defined and takes values in $[0,+\infty]$. Since $\Theta(v,w)= 0$ if and only if $v= tw$ for some $t>0$, we derive  $\Theta$ defines a pseudo-metric on $\mathcal{C}$. In this way,  $\Theta$ induces a metric on a projective quotient space of $\mathcal{C}$ called  the \emph{projective metric of} $\mathcal{C}$.  

We point out that the projective metric depends in a monotone way on the cone: if $\mathcal{C}_1 \subset \mathcal{C}_2$ are two convex cones in $E$, then 
$\Theta_2(v,w) \leq \Theta_1(v,w)$ {for any} $v,w \in \mathcal{C}_1$, 
where $\Theta_1$ and $\Theta_2$ are the projective metrics in $\mathcal{C}_1$ and $\mathcal{C}_2$,  respectively.

Furthermore, note that if $E_1,E_2$ are Banach space, $L:{E}_1 \to {E}_2 $ is a linear operator, and $\mathcal{C}_1, \mathcal{C}_2$ are convex cones in ${E}_{1}, {E}_{2}$, respectively, such that $L(\mathcal{C}_{1}) \subset \mathcal{C}_{2}$, then $\Theta_2(L(v),L(w)) \leq \Theta_1(v,w)$ {for any} $v,w \in \mathcal{C}_1$, where $\Theta_1$ and $\Theta_2$ are the projective metrics in $\mathcal{C}_1$ and $\mathcal{C}_2$,  respectively.

In general,  $L$ need not be a strict contraction, that will be the case for instance if  $L(\mathcal{C}_{1})$ had finite diameter in $\mathcal{C}_{2}$. The next  result will be a key tool to establish the spectral gap for the Ruelle--Perron--Frobenius operator. Its proof can be found in [\cite{VianaStoch}, Proposition 2.3].

\begin{Prop}\label{cont viana}
	Let $\mathcal{C}_{1}$ and $\mathcal{C}_{2}$ be closed convex cones in the Banach spaces ${E}_1$ and ${E}_2$, respectively. If $L:E_1 \to E_2 $ is a linear operator such that $L(\mathcal{C}_{1}) \subset \mathcal{C}_{2}$ and  $\Delta = {\rm diam}_{\Theta_2}(L(\mathcal{C}_{1})) < \infty$, then
	$$\Theta_2 \left(L(\varphi), L(\psi) \right) \leq (1-e^{-\Delta}) \cdot \Theta_1 \left( \varphi, \psi \right) \quad \mbox{for all} \ \varphi, \psi \in \mathcal{C}_{1}.$$
\end{Prop}

%We will restrict our attention to cones of locally H\"older continuous functions. Next, it follows some definitions. 

Our goal is to apply the last result to the transfer operator acting in  a special class of cones. More precisely those of locally H\"older continuous functions that we define as follows.
  
Given $\delta>0$  a function $\varphi$ is said to be  $(C, \alpha)$-H\"older continuous in balls of radius $\delta $ if for some constant $C>0$  we have $$|\varphi(x)-\varphi(y)|\leq Cd(x, y)^\alpha \ \mbox{ for all} \  y \in B(x, \delta).$$
Denote by $|\varphi|_{\alpha,\delta}$ the smallest H\"older constant of $\varphi$ in balls of radius $\delta>0$.

We fix $\delta>0$ and consider for each $k>0$  the convex cone of locally H\"older continuous observables defined on $M$: 
\begin{equation} \label{cone holder}
\mathcal{C}_{k,\delta}= \left\{ \varphi : \varphi>0 \ \mbox{and} \  \frac{|\varphi|_{\alpha,\delta}}{\inf \varphi} \leq k \right\}.
\end{equation}

Considering the classes of cones of locally H\"older continuous observables, it is possible to give a more explicit expression to the projective metric, which we will  be denoted by $\Theta_k$.  The use of this  expression allow to prove that the diameter of  a cone, $\mathcal{C}_{k,\delta}$ is finite, if  $k$ is large enough. We state these results below and refer the reader to \cite{RS17} for their proofs.  

\begin{Lema}\label{metrica cone}[\cite{RS17}, Lemma 4.3]
	The metric $\Theta_k$ in the cone $\mathcal{C}_{k,\delta}$ is given by \[\Theta_k(\varphi , \psi)= \log \left( \frac{B_k(\varphi, \psi)}{ A_k(\varphi, \psi) }\right),\] where  
	$$A_k(\varphi, \psi):=\displaystyle\inf_{d(x,y)< \delta, z\in M} \frac{k|x-y|^{\alpha}\psi(z) - (\psi(x)-\psi(y))}{k|x-y|^{\alpha}\varphi(z) - (\varphi(x) - \varphi(y))} $$
	and $$B_k(\varphi, \psi) := \displaystyle\sup_{d(x,y)< \delta, z\in M} \frac{k|x-y|^{\alpha}\psi(z) - (\psi(x)-\psi(y))}{k|x-y|^{\alpha}\varphi(z) - (\varphi(x) - \varphi(y))}.$$
In particular, we have  $$A_k(\varphi, \psi)\leq \inf_{x\in M}\left\{\frac{\varphi(x)}{\psi(x)}\right\} \quad \mbox{and} \quad B_k(\varphi, \psi)\geq \sup_{x\in M}\left\{\frac{\varphi(x)}{\psi(x)} \right\}.$$
\end{Lema}

\begin{Prop}\label{diam finito}[\cite{RS17}, Proposition  5.3]
The cone $\mathcal{C}_{\hat{\lambda}k,\delta}$ has finite diameter for $k>0$ sufficiently large.
\end{Prop}

We end this section with some technical results regarding H\"older continuous functions that will be used in the sequel.

\begin{Lema}\label{bolas} If 
	$\varphi: M \to \mathbb{R}$ is a $(C, \alpha)$-H\"older continuous function  in  balls of radius $\delta>0$, then there exists $m=m(\delta)>0$ such that $\varphi$ is $(m\cdot C, \alpha)$-H\"older continuous.
\end{Lema}
\begin{proof} By the compactness of $M$, there exists $N\in\mathbb{N}$ which depends only on $\delta$ such that given $x, y\in M$ there are $z_{0}=x, z_{1},\dots , z_{N+1}=y$ with $d(z_{i}, z_{i+1})\leq\delta$ for all $i=0,\dots, N$ and $d(z_{i}, z_{i+1})\leq d(x, y).$
	
	Since $\varphi$ is $(C, \alpha)$-H\"older continuous in balls of radius $\delta$ it follows that
	$$\left |\varphi(x)-\varphi(y)\right |\leq \sum_{i=0}^{N} \left | \varphi(z_{i})-\varphi(z_{i+1}) \right |\leq \sum_{i=0}^{N} C d(z_{i}, z_{i+1})^{\alpha}\leq C(N\!+\!1)d(x, y)^{\alpha}.$$
	
	Therefore,  $\varphi$ is $(m\cdot C, \alpha)$-H\"older continuous where $m=N\!+\!1$.
\end{proof}

\begin{remark}\label{remark.v}
Note that considering balls of radius $({\vi}^N\cdot \de)$, by the proof of  Lemma \ref{bolas}, with $\vi$ satisfying  \eqref{h1}, we can conclude that $\varphi$ is 	$(([\vi^N]+1)\cdot C, \alpha)$-H\"older continuous, where $[\vi^N]$ denotes the  greatest integer less than or equal to $\vi^N$. 
\end{remark}

\begin{Lema}\label{cotasup}
	For each $\varphi \in \mathcal{C}_{k,\delta}$, 
	$$\sup\varphi \leq \inf \varphi \cdot (2\cdot m\cdot d^{\alpha})\cdot k,$$
	where $d$ denotes the diameter of $M$.
\end{Lema}
\begin{proof}
Let  $\varphi \in \mathcal{C}_{k,\delta} $. It follows from Lemma \ref{bolas} that 
	$$	\sup \varphi - \inf\varphi \leq m\cdot |\varphi|_{\alpha,\de}\cdot d^{\alpha} \leq m \cdot \inf \varphi \cdot k \cdot d^{\alpha}.$$
	
	Therefore,	$\sup \varphi \leq \inf \varphi\cdot  (2\cdot m\cdot d^{\alpha})\cdot k.$
\end{proof}

{%\displaystyle\lim_{H \to 0 }   \dfrac{\|  \Gamma(x+H) - \Gamma(x)   - A_{x} (H)  \| _Y}{\| H \|_X}= 0.}

\subsection{Eigenprojections}\label{defprojection}
In this section we state some classical results of spectral theory of bounded operators. 
Let $(X, ||  \cdot || )$ be a complex Banach vector space. Let $\mathcal{B}(X)$ be the space of all the linear operators $T:X\to X$  that are bounded.

Although the next result being classical  we present it  here since its proof  will be useful to obtain our results in the next section. 
\begin{Lema}\label{openproperty}
Let $X$ be a Banach space and let $T \in \mathcal{B}(X)$ be a bounded linear operator. If $T$ is invertible and $\|T - S\| < \|T^{-1}\|^{-1}$, then $S$ is invertible. In particular, the set of invertible
operators is open on  $\mathcal{B}(X)$.
\end{Lema}

\begin{proof}
Since $T$ is invertible and $\|(T - S)T^{-1}\|<  1$ then  $I - (T - S) T^{-1}$ has a bounded inverse given by 
          $$ \sum_{n=0}^{\infty} [(T - S)T^{-1}]^n.$$ 
 Moreover
\begin{align*}
T^{-1}\sum_{n=0}^{\infty} [(T - S)T^{-1}]^n &=T^{-1}[ I - (T - S)T^{-1} ] ^{-1}\\
&=T^{-1}[I - T (T-S)^{-1}]\\&= T^{-1}- (T-S)^{-1}\\&= [T-(T-S)]^{-1} = S^{-1},
\end{align*}
thus $S$ is invertible and  $T^{-1} \sum_{n=0}^{\infty} [(T - S)T^{-1}]^n$ is the inverse of $S$.
%\begin{align*}
%[T - (T - S)]^{-1} & = T^{-1} [I - (T - S)T^{-1}]^{-1}\\
%& =  T^{-1} \sum_{n=0}^{\infty} [(T - S)T^{-1}]^n.
%\end{align*}
%The last expression defines a bounded operator since  $\|S - T\| < \|T^{-1}\|^{-1}$. Moreover, by direct calculation, we can check that this operator is an inverse for $S$, which completes the proof. 
\end{proof}

We point out that  what we present next remains valid even if  the Banach space considered is a real one and not complex. This is because we can consider the complexification of the space and the operator acting on it. 

A bounded linear operator $E:X \to X$ is called a \emph{projection} if it satisfies $E^2 = E$ in which case we can write the following direct sum decomposition
$$    X= \mbox{Im}(E) \oplus \mbox{Ker}(E).$$

Note that this decomposition is such that $x=E(x)+(I-E)(x)$ for all $x\in X,$ where $I$ is the identity map.

\begin{Teo}[Separation of  the spectrum] Let $X$ be a Banach space and $T\in B(X)$. Suppose that the spectrum of $T$ has the following decomposition $Spec(T)=\sigma_1\cup\sigma_2$ where $\sigma_1$ and $\sigma_2$ are disjoint compact sets. If $\gamma$ is a closed smooth simple curve which does not intersect $Spec(T)$  and which contains  $\sigma_1$ in its interior and $\sigma_2$ in its exterior then the operator defined by
$$E:=\frac{1}{2\pi i}\int_{\gamma}(zI-T)^{-1}\, dz$$ 
is a projection and it satisfies:
\begin{enumerate}
\item $ET=TE$ and $\mbox{Ker}(E), \mbox{Im}(E)$ are $T$-invariant;
\item $Spec(T|_{\mbox{Im}(E)})=\sigma_1$ and $Spec(T|_{\mbox{Ker}(E)})=\sigma_2.$
\end{enumerate}
\end{Teo}

For the proof of the last result the reader can consult, for instance, [\cite{Kato}, Theorem 6.17]. We are interested in a particular case of this result when the spectrum admits an isolated point. That is, there exist an eigenvalue $\lambda$ and a closed smooth  simple  curve  $\gamma$ such that $\lambda$  is the unique element of the spectrum in the interior of $\gamma$. In this way we state the next corollary.

\begin{Cor}\label{eigenprojection}
Let  $\lambda \in Spec(T)$ be an isolated eigenvalue and let $\gamma$  be a closed smooth simple curve that separates $\lambda$ from the rest of the spectrum. Then 
 $$ E:=   \frac{1}{2\pi i}\int_{\gamma}(zI-T)^{-1}\, dz  $$ 
is a projection. 
Moreover, $E$ is  the \emph{eigenprojection} of $\lambda$, Im($E$) is the \emph {eigenspace} of $\lambda$ and $dim\left(\mbox{Im}(E)\right)$ is the \emph{geometric multiplicity} of $\lambda$. 
\end{Cor}

%%%%%%%%%%%%%%%%%%%%%%%%%%%%%%%%%%%%%%%%%%%%%%%%%%%%%%%%%%%%%%%%%
%%%%%%%%%%%%%%%%%%%%%%%%%%%%%%%%%%%%%%%%%%%%%%%%%%%%%%%%%%%%%%%%%%%%%%%%%%%%%%%  Spectral Gap     %%%%%%%%%%%%%%%%%%%%%%%%%%%%%%%%%%%%%
%%%%%%%%%%%%%%%%%%%%%%%%%%%%%%%%%%%%%%%%%%%%%%%%%%%%%%%%%%%%%%%%%   
%%%%%%%%%%%%%%%%%%%%%%%%%%%%%%%%%%%%%%%%%%%%%%%%%%%%%%%%%%%%%%%%%

\section{Spectral gap of the transfer operator}
\label{cones}

In this section we prove that the transfer operator admits the spectral gap property when restrict to the space H\"older continuous observables . As previously mentioned, this property is the core for proving  the main results of this paper. We  closely  follow the ideas of \cite{RS17} where it is considered a model of a non-uniformly expanding map where  expansion rates are explicit. In the present work, we consider   a more general class of non-uniformly expanding maps
which contains their example. Still some of their results remain true without major alterations, in which case, we will refer the reader to their proof. 

Let $f \in \mathcal F$. As seen before, each
%Given $\sigma\in(0,1)$,
$x\in \Sigma_{\sigma}$  admits infinitely many hyperbolic times. 
Recall that we fixed $\delta >0$ in Subsection~\ref{pre balls} depending only on $f$ and $\sigma$. We also take
$N$ given by equation~\eqref{ene}. %and $\sigma >0$  given by \ref{sigma}.
%Let $(f, \phi)\in\mathcal{H}_{\sigma}$. 

We assume that $\phi: M \to \mathbb{R}$ is a H\"older continuous potential, hyperbolic for $f$ and satisfying
\begin{equation}\tag{$\ast$ $\ast$} \label{2 estrelas}
\left( e^{N \Var\phi} \left[\vi^N \! +\! 1\right] +\dfrac{2 m d^{\alpha}|e^{N\phi}|_{\alpha}}{e^{N\inf \phi}}\right)\!\!  \left( \dfrac{[(\deg(f)]^N -1) \vi^{N\alpha}}{[\deg(f)]^N} + \dfrac{\gamma^{\alpha}}{[\deg(f)]^N}  \right) < 1
\end{equation}
for some $0<\alpha<1.$

The role of  this assumption will be transparent in the proof of Proposition~\ref{inv cone}. However,\eqref{2 estrelas} can be weaken as described in Remark~\ref{remark theta}.

Note  that for each $\varphi\in C^{0}(M)$ and each $x\in M$  we have
%\begin{Lema}
	%For each $x\in M$ and each $\varphi\in C^{0}(M)$ it holds 
%	$$ \mathcal{L}(\varphi)(x)  \geq [\deg(f)]^N e^{N \inf \phi} \inf\varphi $$
%\end{Lema}
%\begin{proof}
%Indeed, we have
\begin{equation} \label{ineq.deg}
		 \mathcal{L}^N(\varphi)(x) = \sum_{y \, \in \, { {f^{-N}(x)} }} e^{S_N\phi(y)}\varphi(y)  
		\geq[\deg(f)]^N e^{N \inf \phi} \inf\varphi . 
	\end{equation}
%\end{proof}

We recall that  the convex cone of locally H\"older continuous observables is defined by: 
\begin{equation*} 
\mathcal{C}_{k,\delta}= \left\{ \varphi : \varphi>0 \ \mbox{and} \  \frac{|\varphi|_{\alpha,\delta}}{\inf \varphi} \leq k \right\}.
\end{equation*}

Next we prove that for $k$ large enough, the cone $\mathcal{C}_{k,\delta}$ is invariant under the $N$-th iterate of the transfer operator.
%
%
%\begin{Lema}\label{metrica cone}
%	The metric $\Theta_k$ in the cone $\mathcal{C}_{k,\delta}$ is given by \[\Theta_k(\varphi , \psi)= \log \left( \frac{B_k(\varphi, \psi)}{ A_k(\varphi, \psi) }\right),\] where  
%	$$A_k(\varphi, \psi):=\displaystyle\inf_{d(x,y)< \delta, z\in M} \frac{k|x-y|^{\alpha}\psi(z) - (\psi(x)-\psi(y))}{k|x-y|^{\alpha}\varphi(z) - (\varphi(x) - \varphi(y))} $$
%	and $$B_k(\varphi, \psi) := \displaystyle\sup_{d(x,y)< \delta, z\in M} \frac{k|x-y|^{\alpha}\psi(z) - (\psi(x)-\psi(y))}{k|x-y|^{\alpha}\varphi(z) - (\varphi(x) - \varphi(y))}.$$
%\end{Lema}

%Given $(f, \phi)\in \mathcal{H}_c^{\ast}$ let  $\delta >0 $ be as in Lemma~\ref{t hiperbolicos}. 
%Consider a cover of $M$ by balls of radius $\delta$.  By compactness of $M$ there exists a  finite  subcover, namely $\mathcal{P}= \{B_1, B_2, \dots, B_l \}$.   For each $i=1, \dots, l$ there exists a positive integer $N_i$  such that  $f^{N_i}(B_i)= M$. Let $N= N(f)= \max \{N_1, \dots, N_ l\}$ and  fix $\varepsilon < \dfrac{\delta}{2}$. Consider the  operator $L= L (f, \phi, N)$ defined by  
%$$L(\varphi)(x) = \displaystyle\sum_{y \, \in \, { {f^{-N}(x)} \cap {B^{f}_{N}(x, \varepsilon)}}} e^{S_N\phi(y)}\varphi(y).$$
%
%\begin{Lema}
%	For each $x\in M$ and each $\varphi\in C^{0}(M)$ it holds 
%	$$ \mathcal{L}(\varphi)(x)  \geq [\deg(f)]^N e^{N \inf \phi} \inf\varphi $$
%	
%\end{Lema}
%
%
%\begin{proof}
%	Indeed, we have
%	\begin{eqnarray*} 
%		 \mathcal{L}(\varphi)(x) &= & \sum_{y \, \in \, { {f^{-N}(x)} }} e^{S_N\phi(y)}\varphi(y)  
%		\geq[\deg(f)]^N e^{N \inf \phi} \inf\varphi . 
%	\end{eqnarray*}
%\end{proof}

\begin{Prop}\label{inv cone}
There exists $0< \lambda < 1$ such that  
$$ \mathcal{L}^N(\mathcal{C}_{k,\delta}) \subset  \mathcal{C}_{\lambda k,\delta} \subset \mathcal{C}_{k,\delta} \quad \mbox{for k large enough.}  $$
\end{Prop}

\begin{proof}
Let $\varphi \in \mathcal{C}_{k,\delta}$.  Thus  $\varphi>0$ and, 
by definition, we  have  $ \mathcal{L}^N(\varphi)>0$. Since $ \mathcal{L}^N$ is a bounded operator we derive that 
$ \mathcal{L}^N(\varphi) $ is continuous. 

In order to prove that $ \mathcal{L}(\varphi) \in \mathcal{C}_{\lambda k,\delta}$  we must show that 
     $$\frac{| \mathcal{L}^N(\varphi)|_{\alpha, \delta}}{\inf  \mathcal{L}^N(\varphi)}  \leq \lambda k  \quad \mbox{for some} \quad 0<\lambda <1.$$

Set $N> 2\n$, where $\n$ is given by Lemma \ref{lemmahyper}. Given $x,y\in M$ satisfying $d(x,y)<\de$, we denote by $x_j,y_j$, $1\leq j\leq [\deg (f)]^N$, the pre-images of $x$ and $y$ under $f^N$, respectively. 

By the definition of the operator $\mathcal{L}^N$ and the constant $|\mathcal{L}^N(\varphi)|_{\alpha, \delta}$ we obtain, by using triangle inequality, the following
\begin{align}\label{aux2}
&\dfrac{|\mathcal{L}^N(\varphi)|_{\alpha, \delta}}{\inf \mathcal{L}^N(\varphi)}  \leq \sup_{d(x,y)<\de}\dfrac{|\mathcal{L}^N(\varphi(x)) - \mathcal{L}^N(\varphi(y))|}{\inf \mathcal{L}^N(\varphi) \, d(x,y)^{\alpha}} \nonumber\\
&  \leq \dfrac{\displaystyle \sum_{j=1}^{{[\deg(f)}]^N} \left| e^{S_N\phi(x_j)} \varphi(x_j) -e^{S_N\phi(y_j)} \varphi(y_j) \right| }{\inf \mathcal{L}^N(\varphi) \, d(x,y)^{\alpha}}\nonumber\\
&\leq \dfrac{\displaystyle \sum_{j=1}^{{[\deg(f)}]^N} \left| e^{S_N\phi(x_j)} \right| |\varphi(x_j) - \varphi(y_j)|}{\inf \mathcal{L}^N(\varphi) \, d(x,y)^{\alpha}} + \dfrac{\displaystyle \sum_{j=1}^{{[\deg(f)}]^N} \left| \varphi(y_j) \right| \left|e^{S_N\phi(x_j)} - e^{S_N\phi(y_j)} \right| }{\inf \mathcal{L}^N(\varphi)  \, d(x,y)^{\alpha}}.
\end{align}

From the definition of the supremum of functions, it follows that the last sums are bounded from above by
\begin{equation*}
 \dfrac{e^{N\sup\phi}  \displaystyle \sum_{j=1}^{{[\deg(f)}]^N}  |\varphi(x_j) - \varphi(y_j)|}{\inf \mathcal{L}^N(\varphi)  \, d(x,y)^{\alpha}} + \sup \varphi   \dfrac{\displaystyle \sum_{j=1}^{{[\deg(f)}]^N}  \left|e^{S_N\phi(x_j)} - e^{S_N\phi(y_j)} \right| }{\inf \mathcal{L}^N(\varphi)  \, d(x,y)^{\alpha}}.
\end{equation*}

On the other hand, by Remark \ref{remark.v} and  equation~\eqref{ineq.deg}, the previous sums can be bounded from above by
\begin{equation}\label{effective}
 \dfrac{e^{N\sup\phi}  [\vi^N + 1]  |\varphi|_{\alpha,\de}  \displaystyle \sum_{j=1}^{{[\deg(f)}]^N} d(x_j,y_j)^{\alpha}} {[\deg(f)]^N  e^{N\inf \phi}  \inf \varphi  \, d(x,y)^{\alpha}} + \dfrac{ \sup \varphi |e^{N\phi}|_{\alpha}  \displaystyle \sum_{j=1}^{{[\deg(f)}]^N} d(x_j,y_j)^{\alpha}} {[\deg(f)]^N  e^{N\inf \phi}  \inf \varphi  \, d(x,y)^{\alpha}}.
\end{equation}

\vspace{0.4cm}
Now, by Lemma \ref{cotasup} we conclude that the expression \ref{effective} is less or equal than
\begin{equation*}
\dfrac{e^{N\sup\phi}  [\vi^N  + 1]  |\varphi|_{\alpha,\de}  \displaystyle \sum_{j=1}^{{[\deg(f)}]^N} d(x_j,y_j)^{\alpha}} {[\deg(f)]^N  e^{N\inf \phi}  \inf \varphi \, d(x,y)^{\alpha}}
+ \dfrac{  2  m {d}^{\alpha} k \inf \varphi  \, |e^{N\phi}|_{\alpha}  \displaystyle \sum_{j=1}^{{[\deg(f)}]^N} d(x_j,y_j)^{\alpha}} {[\deg(f)]^N  e^{N\inf \phi}  \inf \varphi  \, d(x,y)^{\alpha}}.
\end{equation*}
Let $\Var\phi= \sup \phi - \inf \phi$. Since $\varphi \in \mathcal{C}_{k,\delta}$, we can  rewrite the expression above as
\begin{equation}\label{aux}
\dfrac{e^{N \Var\phi}   [\vi^N + 1] k   \displaystyle \sum_{j=1}^{{[\deg(f)}]^N} d(x_j,y_j)^{\alpha}} {[\deg(f)]^N  \, d(x,y)^{\alpha}} + \dfrac{ 2 m {d}^{\alpha} k\ |e^{N\phi}|_{\alpha}  \displaystyle \sum_{j=1}^{{[\deg(f)}]^N} d(x_j,y_j)^{\alpha}} {[\deg(f)]^N  e^{N\inf \phi}  \, d(x,y)^{\alpha}}.
\end{equation}

\vspace{0.4cm}
Rearranging the indexes, if necessary, we can suppose that $$d(x_j,y_j)\leq  \vi^{N}\cdot d(x,y)\ \mbox{ for} \ 1\leq j\leq [\deg (f)]^N.$$  Moreover, according to \ref{hyper},  there exists at least one  index, let us call it $M$, such that
\begin{equation}\label{aux1}
 d(x_M, y_M) \leq \gamma \cdot d(x,y).
\end{equation}

Hence, the sums \ref{aux} are bounded from above by
\begin{align}\label{aux3}
&\dfrac{e^{N \Var\phi}  [\vi^N +1]  k \left[ [(\deg(f)]^N -1)\vi^{N\alpha} + \gamma^{\alpha}  \right]  d(x,y)^{\alpha}} {[\deg(f)]^N  \, d(x,y)^{\alpha}}\nonumber \\
& + \dfrac{ 2  m {d}^{\alpha} k |e^{N\phi}|_{\alpha}  \left[ [(\deg(f)]^N -1)  \vi^{N\alpha} + \gamma^{\alpha}  \right] d(x,y)^{\alpha}}{[\deg(f)]^N  e^{N\inf \phi}  \, d(x,y)^{\alpha}}.
\end{align}

From \eqref{aux2}, \eqref{aux1} and \eqref{aux3}, it follows that
\begin{align*}
\dfrac{|\mathcal{L}^N(\varphi)|_{\alpha, \delta}}{\inf \mathcal{L}^N(\varphi)}& \leq \left(\dfrac{e^{N \Var\phi} \left[\vi^N + 1\right]  \left[ [(\deg(f)]^N -1) \vi^{N\alpha} + \gamma^{\alpha}  \right]}{[\deg(f)]^N} 
\right.\\ & \left. + \dfrac{2 m  d^{\alpha}|e^{N\phi}|_{\alpha} \left[ [(\deg(f)]^N -1) \vi^{N\alpha} +  \gamma^{\alpha} \right]}{[\deg(f)]^N e^{N\inf \phi}}  \right) k.
\end{align*}

Therefore, we conclude that
$$\dfrac{|\mathcal{L}^N(\varphi)|_{\alpha, \delta}}{\inf \mathcal{L}^N(\varphi)}  \leq  \hat{\lambda} k,$$
where 
 \begin{align*}
\hat{\lambda}  = \left( e^{N \Var\phi} \left[\vi^N \! +\! 1\right] +\dfrac{2 m d^{\alpha}|e^{N\phi}|_{\alpha}}{e^{N\inf \phi}}\right)\!\!  \left( \dfrac{[(\deg(f)]^N -1) \vi^{N\alpha}}{[\deg(f)]^N} + \dfrac{\gamma^{\alpha}}{[\deg(f)]^N}  \right).
%\cdot \left[ [(\deg(f)]^N -1)\cdot \vi^{N\alpha} + \gamma^{\alpha}  \right]}{[\deg(f)]^N} 
%\right.\\ & \left. + \dfrac{2\cdot m \cdot d^{\alpha}|e^{N\phi}|_{\alpha} \left[ [(\deg(f)]^N -1)\cdot \vi^{N\alpha} +  \gamma^{\alpha} \right]}{[\deg(f)]^N \cdot e^{N\inf \phi}}  \right)\cdot k.
\end{align*}
By condition~\eqref{2 estrelas}  we have  $\hat{\lambda} <1$, which completes the proof.
\end{proof}

\begin{remark}\label{remark theta}
If we assume that  $\vartheta^N < 2$, then estimate \eqref{effective} can be improved. Indeed, this extra assumption  allows us to  replace $$|\varphi(x_j) - \varphi (y_j)| \leq [\vartheta^N + 1] |\varphi|_{\alpha, \delta} d(x_j, y_j)^{\alpha}$$  by the following  $$|\varphi(x_j) - \varphi (y_j)| \leq [(\vartheta^N -1 )^{\alpha}+ 1]\,  |\varphi|_{\alpha, \delta} \, d(x_j, y_j)^{\alpha}.$$
Now replacing $ [\vartheta^N + 1]$ by $ [(\vartheta^N -1 )^{\alpha}+ 1] $ in the assumption \eqref{2 estrelas}, we  enlarge the class of  potentials that satisfy this assumption.

 We point out that the extra assumption  $\vartheta^N < 2$ is in fact quite general.  For instance, it is satisfied by  the examples presented in Section~\ref{exemplos} (as can be seen  in  \cite{BCV} for Example 8.1  and  in \cite{RS17} for  Example 8.3).

Moreover, note that condition \eqref{2 estrelas} could be improved if we controlled the visits of the orbits to the complement  of the non-uniformly expanding set. 
\end{remark}

The last Proposition shows that the cone $\mathcal{C}_{k,\delta}$ is invariant under $\mathcal{L}_{f, \phi}^{N}$, moreover,  Proposition~\ref{diam finito} ensures that $\mathcal{C}_{k,\delta}$  has finite diameter. Therefore  Proposition~\ref{cont viana}
implies the next result. 

\begin{Prop} \label{contracao no cone}
The operator $\mathcal{L}_{f, \phi}^{N}$ is a contraction in the cone $\mathcal{C}_{k,\delta}$: for the constant $\Delta = \mbox{diam}(\mathcal{C}_{\hat{\lambda}k,\delta}) >0$ we have
   $$\Theta_{k} \left(\mathcal{L}_{f, \phi}^{N}(\varphi), \mathcal{L}_{f, \phi}^{N}(\psi) \right) \leq (1-e^{-\Delta}) \cdot \Theta_k \left( \varphi, \psi \right) \quad \mbox{for all} \ \varphi, \psi \in \mathcal{C}_{k, \delta}.$$	
\end{Prop}

Let $\lambda_{f, \phi}$ be the spectral radius of the transfer operator $\mathcal{L}_{f,\phi}$. The existence of a probability measure $\nu_{f, \phi}$ satisfying $\mathcal{L}_{f, \phi}^{\ast}\nu_{f, \phi}=\lambda_{f, \phi}\nu_{f, \phi}$ and $\nu_{f, \phi}({\Sigma}_{\sigma})=1$ was proven in  \cite{RV16}. Moreover, \cite{RV16} also guarantees that $\log \lambda_{f, \phi} =  P_{f}(\phi)$.
From the last proposition we will obtain the existence of an eigenfunction $h_{f, \phi}$ of $\mathcal{L}_{f,\phi}$ associated to the spectral radius. 

\begin{Prop} \label{h} There exists a H\"older continuous function $h_{f, \phi}:M\to\mathbb{R}$ bounded away from zero and infinity which satisfies  
$\mathcal{L}_{f,\phi}h_{f,\phi}=\lambda_{f,\phi} h_{f,\phi}.$
\end{Prop}

\begin{proof} Define $L=\lambda_{f, \phi}^{-N}\mathcal{L}^{N}_{f, \phi}$ and consider the sequence $\{L^n(\bold{1})\}_{n\in \mathbb{N}}$. Since $\nu_{f, \phi}$ is an eigenmeasure associated to $\lambda_{f, \phi}$, we have for every $n\geq 1$
\begin{eqnarray*} 
\int L^{n}(\bold{1})\ d\nu_{f, \phi}= \int \lambda^{-nN}\mathcal{L}^{nN}_{f, \phi}(\bold{1}) \ d\nu_{f, \phi} &=& \lambda^{-nN}\int \bold{1} \ d(\mathcal{L}^{nN}_{f, \phi})^{\ast}\nu_{f, \phi}\\
& = & \int \bold{1}\ d\nu_{f, \phi} =1.
\end{eqnarray*}
Thus each term of the sequence satisfies $\sup  L^n(\bold{1})\geq 1$ and $\inf L^n(\bold{1})\leq 1$. 
Note that  $\bold{1} \in C_{k,\delta}$ and  by Proposition~\ref{inv cone}, the cone    
 $\mathcal{C}_{k,\delta}$ is invariant under~$L$, then $\{L^n(\bold{1})\}$ is a sequence in  $\mathcal{C}_{k,\delta}$. By Lemma~\ref{cotasup} every $\varphi\in \mathcal{C}_{k,\delta}$ satisfies $\sup\varphi \leq \inf \varphi \cdot (2\cdot m\cdot d^{\alpha})\cdot k,$ therefore we conclude that $\{L^n(\bold{1})\}$  is uniformly bounded away from zero and infinity by
$$\left(2\cdot m\cdot d^{\alpha}\cdot k\right)^{-1}\leq\inf L^n(\bold{1})\leq 1\leq \sup L^n(\bold{1}) \leq 2\cdot m\cdot d^{\alpha}\cdot k. $$

Moreover, since $L^n(\bold{1})$ is $\alpha$-H\"older continuous in balls of radius $\delta$ for all $n\geq 1$, Lemma~\ref{bolas}  implies that $L^n(\bold{1})$ is an $\alpha m$-H\"older continuous function.

Now we prove that $\left\{ L^n(\bold{1})\right\}$ is a Cauchy sequence in the sup norm. Let $\Delta= \diam(C_{k,\delta})$ and $\tau= 1-e^{-\Delta}$. Proposition~\ref{contracao no cone} implies that for every $j, l\geq n$ the projective metric satisfies
    $$\Theta_{k}(L^j(\bold{1}), L^l (\bold{1})) \leq \Delta \tau^n .  $$

According to  Lemma~\ref{metrica cone} we can write  $$\Theta_{k}(L^j(\bold{1}), L^l (\bold{1}))= \log\left(\frac{B_k(L^j(\bold{1}), L^l (\bold{1}))}{A_k(L^j(\bold{1}), L^l (\bold{1}))}\right),$$
and combining with the last inequality we obtain
\begin{eqnarray*} \label{cota sup}
 e^{-\Delta \tau^{n}}\!\!\leq A_{k}(L^j(\bold{1}), L^l(\bold{1}))\!\!\!&\leq&\!\!\!\inf \frac{L^j(\bold{1})}{L^l(\bold{1})}\\
\!\!\!&\leq&\!\!\! 1\\
\!\!\!&\leq&\!\!\! \sup \frac{L^j(\bold{1})}{L^l(\bold{1})} \leq B_{k}(L^j(\bold{1}),L^l(\bold{1})) \leq e^{\Delta \tau^{n}}.
\end{eqnarray*}
Note that second  and fifth inequalities follow from the second part of Lemma~\ref{metrica cone}.

Then for all $j, l\geq n$, we have:
$$\left\|  L^j(\bold{1}) - L^l(\bold{1}) \right\|_{0} \leq \left\| L^l(\bold{1}) \right\|_{0} \left\| \frac{L^j(\bold{1})}{L^l(\bold{1})} -1 \right\|_{0} \leq \left\|L^l(\bold{1}) \right\|_{0} \left( e^{\Delta \tau^n} -1 \right) \leq \tilde{R}\Delta\tau ^n, $$
which proves that $\left\{ L^n(\bold{1})\right\}$ is a Cauchy sequence.
Therefore $\{ L^n(\bold{1})\}$ converges uniformly to a function $h_{f, \phi}:M\to\mathbb{R}$ in the cone $\mathcal{C}_{k, \delta}$ and consequently $\alpha m$-H\"older continuous and bounded away from zero and infinity. 

It remains to check that $\mathcal{L}_{f,\phi}h_{f,\phi}=\lambda_{f,\phi} h_{f,\phi}.$ Notice that if we replace in the definition of the sequence the function $\bold{1}$
by $\lambda_{f, \phi}^{-1}\mathcal{L}_{f, \phi}(\bold{1})$, a similar argument shows that the sequence  $\{L^n(\lambda_{f, \phi}^{-1}\mathcal{L}_{f, \phi}(\bold{1}))\}$ converges to $h_{f, \phi}.$ From the continuity of the transfer operator we conclude that
\begin{eqnarray*}
\mathcal{L}_{f, \phi}(h_{f, \phi})&=&\mathcal{L}_{f, \phi}\left(\lim L^{n}(\bold{1})\right)=\mathcal{L}_{f, \phi}\left(\lim \lambda_{f, \phi}^{-nN}\mathcal{L}_{f, \phi}^{nN}(\bold{1})\right)\\
&=&\lim \mathcal{L}_{f, \phi}(\lambda_{f, \phi}^{-nN}\mathcal{L}_{f, \phi}^{nN})(\bold{1})=\lambda_{f, \phi}\lim \lambda_{f, \phi}^{-nN-1}\mathcal{L}_{f, \phi}^{nN+1}(\bold{1})\\
&=&\lambda_{f, \phi}\lim \lambda_{f, \phi}^{-nN}\mathcal{L}_{f, \phi}^{nN}\left(\lambda_{f, \phi}^{-1}\mathcal{L}_{f, \phi}(\bold{1})\right)\\
&=&\lambda_{f, \phi}h_{f, \phi}.
\end{eqnarray*}
 \end{proof}

Proposition~\ref{contracao no cone} implies the next result. Its proof is analogous to   [\cite{RS17}, Proposition 5.5] and therefore it will be omitted  here.
\begin{Prop} \label{cota norma}
	Let $(f, \phi) \in \mathcal{H}_{\sigma}$. There exist  a constant $R>0$ and $0<\tau <1$ such that for every $\varphi \in \mathcal{C}_{k, \delta}$ satisfying $\int \varphi \ d\nu_{f, \phi} = 1$ we have 
	$$\left\| \lambda_{f, \phi} ^{-n}\mathcal{L}^n_{f,\phi}(\varphi) - h_{f, \phi} \right\|_{\alpha} % = \left\|  \lambda^{-n}_{f, \phi} \mathcal{L}^n_{f, \phi} (\varphi) - h_{f, \phi}  \right\|_{0}  + \left|   \lambda_{f, \phi} ^{-n}\mathcal{L}^n_{f, \phi} (\varphi) - h_{f, \phi} \right| _{\alpha, \delta} 
\leq R\tau^n   \quad \forall n\geq 1.$$ 
\end{Prop}

We finish this section by proving the spectral gap property of the transfer operator.

\begin{Teo}\label{gap} For $(f, \phi)\in\mathcal{H}_{\sigma},$ the spectrum of the operator $\mathcal{L}_{f, \phi}$, acting on the space $C^{\alpha}\left(M\right),$ has a decomposition: there exists $0 < r <\lambda_{f, \phi}$ such that $Spec(\mathcal{L}_{f, \phi}) = \left\{ \lambda_{f, \phi} \right\} \cup \Sigma $ with $\Sigma$ contained in a ball $B(0, r)$ centered at zero and of radius $r$.
\end{Teo}

\begin{proof} Let $\mathcal{L}=\lambda_{f, \phi}^{-1}\mathcal{L}_{f, \phi}$ be the normalized operator. Consider the space $E_0 = \left\{ \psi \in C^{\alpha}\left(M\right) : \int \psi \ d\nu_{f, \phi} = 0 \right\}$ and let $E_1$ be the eigenspace of dimension $1$ of $\mathcal{L}$  associated to the eigenvalue $1.$ We point out that  ${\rm dim} \, E_1=1$ since $h_{f, \phi}$ is the unique eigenfunction associated to the spectral radius. Notice that it is possible to decompose $C^{\alpha}\left(M\right)$ as a direct sum of $E_0$ and $E_1$ by writing any $\varphi \in C^{\alpha}\left(M\right)$ as follows 
$$\varphi = \left[ \varphi - \left( \int\varphi \ d \nu_{f, \phi}\right) \cdot h_{f, \phi}\right] + \left[ \left( \int \varphi \ d\nu_{f, \phi} \right)\cdot h_{f, \phi} \right]= \varphi_0 + \varphi_1$$
with $\varphi_0 \in E_0$  and $\varphi_1 \in E_1$.

In fact, since $\int h_{f, \phi}\ d\nu_{f, \phi}=1$, we derive that  $\varphi_0:= \left[ \varphi - \int\varphi \ d \nu_{f, \phi}\cdot h_{f, \phi}  \right]$ satisfies 
\begin{align*}
\int\varphi_0\ d\nu_{f, \phi}&=\int \left(\varphi - \int\varphi \ d \nu_{f, \phi}\cdot h_{f, \phi}\right) \ d\nu_{f, \phi}\\
&= \int\varphi\ d \nu_{f, \phi} - \int\varphi \ d \nu_{f, \phi}\cdot\int h_{f, \phi}\, d \nu_{f, \phi} \\
& =0
\end{align*}
and so it is an element of $E_0$. Moreover $\varphi_1 = \left[ \left( \int \varphi \ d\nu_{f, \phi} \right)\cdot h_{f, \phi} \right]$ belongs to $E_1$ because
\begin{align*}
\mathcal{L}(\varphi_1)&=\lambda_{f, \phi}^{-1}\mathcal{L}_{f, \phi}(\varphi_1)=\lambda_{f, \phi}^{-1}\mathcal{L}_{f, \phi}\left(\int \varphi \ d\nu_{f, \phi}\cdot h_{f, \phi}\right)\\
&=\lambda_{f, \phi}^{-1}\int \varphi \ d\nu_{f, \phi}\cdot\mathcal{L}_{f, \phi}(h_{f, \phi})
=\lambda_{f, \phi}^{-1}\int \varphi \ d\nu_{f, \phi}\cdot \lambda_{f, \phi} h_{f, \phi}=\varphi_1.
\end{align*}

Now it is enough to show that $\mathcal{L}^n$ is a contraction in $E_0$ for $n$ sufficiently large.

Fix $k>0$ large enough. Given $\varphi \in E_0$ with $\left|\varphi \right|_{\alpha, \delta} \leq 1$ notice that $\varphi$ does not necessarily belong to the cone $\mathcal{C}_{k, \delta}$ but for example $\left(\varphi +2\right) \in \mathcal{C}_{k, \delta}$ since $$\frac{\left|\varphi + 2 \right|_{\alpha. \delta}}{\inf\left(\varphi +2 \right)}  = \frac{\left|\varphi \right|_{\alpha. \delta}}{\inf\left(\varphi +2 \right)} \leq  \frac{1}{\inf\left(\varphi +2 \right)} \leq k \quad \mbox{for} \  k  \ \mbox{large}  . $$
Therefore applying Proposition~\ref{cota norma} we derive that
\begin{eqnarray*} 
	\left\| \mathcal{L}^n(\varphi) \right\|_{\alpha} &=& \left\| \mathcal{L}^n(\varphi + 2) - \mathcal{L}^n(2)\right\|_{\alpha}\\ \\
	&\leq& \left\| \mathcal{L}^n(\varphi + 2) - 2h_{f, \phi} \right\|_{\alpha} + \left\| \mathcal{L}^n(2) - 2h_{f, \phi} \right\|_{\alpha} \\\\
	&\leq& \left\| \left( \int \varphi + 2 \ d\nu_{f, \phi} \right) \mathcal{L}^n\left(\frac{\varphi + 2}{\int \varphi + 2 \ d\nu_{f, \phi}}\right) -2h_{f, \phi} \right\|_\alpha\\ \\
	 &+&\left\| \mathcal{L}^n(2) - 2h_{f, \phi} \right\|_\alpha \\\\
	&\leq& 2 \left\| \mathcal{L}^n\left(\frac{\varphi + 2}{\int \varphi + 2 \ d\nu_{f, \phi}}\right) -h_{f, \phi} \right\|_\alpha + 2\left\| \mathcal{L}^n( \mathbf{1}) - h_{f, \phi} \right\|_\alpha \\\\
	&\leq& 2 L\tau^n + 2L\tau^n = 4L\tau^n .
\end{eqnarray*}

This contraction shows that the spectrum of $\mathcal{L}$ admits a decomposition $Spec(\mathcal{L})=\{1\}\cup\Sigma_{0}$ where $\Sigma_{0}$ is contained in a ball centered at zero and radius strictly less than one. To conclude the proof just observe that we obtain the spectrum of $\mathcal{L}_{f, \phi}$ by multiplying the spectrum of $\mathcal{L}$ by $\lambda_{f, \phi}$.
\end{proof}

The spectral gap property  established in the previous result is the key tool we will use to derive  our main results. In the following sections we prove these results employing  probabilistic and analytic arguments. 
In \cite{Benoit} the authors use  a differential geometrical approach to obtain consequences of the spectral gap property for a  rather general framework. Some of our results could also be obtained 
using \cite{Benoit}. More specifically, Theorem~\ref{TCL}  follows from \cite[Theorem 5.6]{Benoit}, while Theorem~\ref{analiticidade potencial}  follows from \cite[Corollary B]{Benoit}.

% However, it is possible to apply results from \cite{Benoit} to obtain some of our results,  namely: The Central Limit Theorem [Theorem 5.6, \cite{Benoit}]; Theorem III and Theorem VI [Corollary B, \cite{Benoit}].
%In the aforementioned work  the authors use  a differential geometrical approach to obtain consequences for the spectral gap property for a quite general framework.
%%%%%%%%%%%%%%%%%%%%%%%%%%%%%%%%%%%%%%%%%%%%%%%%%%%%%%%%%%%%%%%%%%%%%%%%%%%%%%%%%%%%%%%%%%%%%%%%%%%%%%%%%% Decaimento e TCL %%%%%%%%%%%%%%%%%%%%%%%%%%%%%%%%%%% 
%%%%%%%%%%%%%%%%%%%%%%%%%%%%%%%%%%%%%%%%%%%%%%%%%%%%%%%%%%%%%%%%%%%%%%%%%%%%%%%%% 
                   
\section{Statistical behavior of the equilibrium state}\label{Statistical}

In this section we conclude the proof of Theorem~\ref{formalismo} and derive  statistical properties of the equilibrium state, namely Theorem~\ref{decaimento f} and Theorem~\ref{TCL}. We show that  a classical proof of the exponential decay  of correlations holds in this context, for the sake of completeness. We end the section recalling that  a central limit theorem can be  obtained  by applying the well known Gordin Theorem. 
%
%In this section we prove Theorem~\ref{decaimento f}, Theorem~\ref{TCL}, and conclude the proof of Theorem~\ref{formalismo}.

Let $\mu_{f, \phi}:=h_{f, \phi}\nu_{f, \phi}$, where $\mathcal{L}_{f,\phi}h_{f,\phi}=\lambda_{f,\phi} h_{f,\phi}$ and  $\mathcal{L}_{f,\phi}^{\ast}\nu_{f,\phi}=\lambda_{f,\phi} \nu_{f,\phi}$. 
It  is straightforward to check that $\mu_{f, \phi}$ is invariant under $f$. Moreover, since $\nu_{f, \phi}(\Sigma_{\sigma})=1$, we also have that $\mu_{f, \phi}(\Sigma_{\sigma})=1.$  

Our goal is to prove that $\mu_{f, \phi}$ is the unique equilibrium state of $(f,\phi)$, which finishes the proof of Theorem~\ref{formalismo}. However, we first establish that the decay of correlations is exponential for the probability measure  $\mu_{f, \phi}$.  For this we 
can  borrow some ideas from \cite{RS17} since we have already  proved that the transfer operator has a spectral gap property.

%The exponential convergence of the transfer operator to the eigenfunction obtained in the previous section  allows us to obtain an exponential decay of correlations for the measure $\mu_{f, \phi}$ associated to $(f, \phi)\in\mathcal{H}_{\sigma}$, as set out in the sequel.

\begin{Prop}\label{decaimento}
	For every $(f, \phi)\in\mathcal{H}_{\sigma}$ the invariant measure $\mu_{f, \phi}$ has exponential decay of correlations for H\"older continuous observables: there exists $0< \tau < 1$ such that for all $\varphi \in L^1(\mu_{f, \phi})$ and $\psi \in C^{\alpha}(M) $ there exists a positive constant $K(\varphi, \psi) $ satisfying:
	$$\left|\int \left(\varphi \circ f^n\right)\cdot \psi \ d\mu_{f, \phi} - \int \varphi \ d\mu_{f, \phi} \int \psi\ d\mu_{f, \phi}  \right| \leq K(\varphi, \psi) \tau^n \quad \mbox{for all} \ n\geq 1.  $$
\end{Prop}

\begin{proof}
	Let $\varphi, \psi\! \in\! \! C^{\alpha}(M)$ and note that the transfer operator satisfy the following for all $n\!\in \! \mathbb{N}$
	$$   \mathcal{L}^ n\left( (\varphi \circ f^n) \cdot\psi \right) = \varphi \cdot \mathcal{L}^ n(\psi).$$ 
	%Let $h_{f, \phi}$ be the eigenfunction of the transfer operator associated to the spectral radius $\lambda_{f, \phi}.$
	Recall that by Proposition~\ref{h}  the eigenfunction of the transfer operator $h_{f, \phi}$ is bounded away from zero and infinity.
	We first assume that 
	%Given $\psi \in C^{\alpha}(M) $ suppose that 
	$\psi \cdot h \in \mathcal{C}_{k, \delta}$ for $k$ large enough and  without loss of generality, we can consider $\int \psi \  d\mu_{f, \phi} = 1$. % Hence, for any $\varphi \in L^1(\mu_{f, \phi})$  we have that
	\begin{eqnarray*}
		&&\left| \int \left(\varphi \circ{ f}^{n}\right)\cdot \psi \ d\mu_{f, \phi}-\int \varphi \ d\mu_{f, \phi} \int \psi\ d\mu_{f, \phi}\right|\\
		&=& \left| \int\! \varphi \cdot \lambda^{-n}_{f, \phi}\mathcal{L}_{f, \phi}^n\left( \psi \cdot h_{f, \phi} \right) \ d\nu_{f, \phi} - \int\! \varphi \ d\mu_{f, \phi}   \right| \\
		&\!\!=\!\!& \int\! \varphi \cdot \left[\frac{\lambda^{-n}_{f, \phi}\mathcal{L}_{f, \phi}^n\left( \psi \cdot h_{f, \phi} \right)}{h_{f, \phi}} - 1 \right] \ d\mu_{f, \phi} \\
		&\!\!\leq\!\!& \int\! \left| \varphi \right| \ d\mu_{f, \phi} \cdot \left\| \frac{\lambda^{-n}_{f, \phi}\mathcal{L}_{f, \phi}^n\left( \psi \cdot h_{f, \phi} \right)}{h_{f, \phi}} - 1 \right\|_0.  
	\end{eqnarray*}

	Therefore, applying Proposition~\ref{cota norma}, there exists some positive constant $L_1$ such that 
	$$\left\| \frac{\lambda^{-n}_{f, \phi}\mathcal{L}_{f, \phi}^n\left( \psi \cdot h_{f, \phi} \right)}{h_{f, \phi}} - 1 \right\|_0 \leq \left\|h_{f, \phi} \right\|_0 \left\| \lambda^{-n}_{f, \phi}\mathcal{L}_{f, \phi}^n\left( \psi \cdot h_{f, \phi} \right) - h_{f, \phi} \right\|_0 \leq  L_1 \tau^ n. $$
	
	Now if $\psi \cdot h \notin \mathcal{C}_{k, \delta}$ we fix $B=k^{-1}|\psi \cdot h |_{\alpha, \delta}$ and consider
	$\xi:=\psi\cdot h$ where
	$$\xi=\xi_{B}^{+}-\xi_{B}^{-}\,\,\,\, \mbox{\and}\,\,\,\, \xi_{B}^{\pm}=\frac{1}{2}\left(|\xi|\pm\xi\right)+B.$$
	
	Thus $\xi_{B}^{\pm}\in \mathcal{C}_{k,\delta}$ and then we apply the previous estimates to $\xi_{B}^{\pm}.$ The result follows by linearity.
\end{proof}
}
%
%\marginpar{Podemos tirar se acrescentarmos aquele par\'agrafo..}
%The reader can observe that the proof of  Proposition~\ref{decaimento} follows similar steps from [\cite{RS17}, Theorem 6.1]. Despite  the similarity to the  mentioned result, we provide the details  for the reader's convenience.

We recall that a measure that is invariant under a map $f$  is called {\it exact} if  it satisfies
$ \lim\limits_{n\to +\infty}\mu(f^n(A))= 1 \  $
for all measurable set $ A$ such that $ \mu(A)>0.$
In particular, an exact probability measure is ergodic. 
%{\Large{\snew{Colocar a defini\c c\~ao de medida exata.}}}

The exponential decay of correlations implies the exactness property. The reader can see  a proof of this result  in [\cite{RS17}, Corollary 6.2].

\begin{Cor} \label{exactness f}
	The invariant measure $\mu_{f, \phi}$ is exact. 
\end{Cor}

{Now we can prove that $\mu_{f, \phi}:=h_{f, \phi}\nu_{f, \phi}$ is the equilibrium state associated to $(f, \phi)\in\mathcal{H}_{\sigma}$ and complete the proof of Theorem~\ref{formalismo}.}

	In \cite{RV16} it was proved that $\nu_{f, \phi}$ satisfies a type of Gibbs property at hyperbolic times: for $\varepsilon\leq\delta$   there exists $C=C(\varepsilon)>0$ such that if  $n$ is a hyperbolic time for $x\in M$ then
	\begin{eqnarray*} \label{Gibbs}
		C^{-1} \leq  \frac{\nu_{f, \phi}(B_\varepsilon(x, n))}{\exp(S_n\phi(y) - n\log\lambda_{f,\phi})}  \leq C.
	\end{eqnarray*}
for all $y\in B_\varepsilon(x, n).$ 
	Recalling that the density $h_{f, \phi}$ is bounded away from zero, it follows that $\mu_{f, \phi}$ is equivalent to $\nu_{f, \phi}$ and, thus, $\mu_{f, \phi}$  also satisfies the Gibbs property at hyperbolic times: 
	\begin{eqnarray*} 
		\tilde{C}^{-1} \leq  \frac{\mu_{f, \phi}(B_\varepsilon(x, n))}{\exp(S_n\phi(y) - n\log\lambda_{f,\phi})}  \leq \tilde{C}.
	\end{eqnarray*}
Rewriting the inequalities above we have for $\mu_{f, \phi}$-almost every point $x\in M$
\begin{eqnarray*}
\log\lambda_{f, \phi}-\lim_{n\to\infty}S_n\phi(y)& \leq& \limsup_{n\to\infty}-\frac{1}{n}\log\mu_{f, \phi}(B_\varepsilon(x, n))\\
&\leq& \log\lambda_{f, \phi}-\lim_{n\to\infty}S_n\phi(y),
\end{eqnarray*}
where the limit was considered when $n$ goes to infinity since $\mu_{f, \phi}$-almost every point $x\in M$ admit infinitely many hyperbolic times.
	From Birkhoff's Ergodic Theorem we obtain that
\begin{eqnarray*}	
\log\lambda_{f, \phi}-\int \phi \,d\mu_{f, \phi}&\leq& \limsup_{n\to\infty}-\frac{1}{n}\log\mu_{f, \phi}(B_\varepsilon(x, n))\\
&\leq &\log\lambda_{f, \phi}-\int \phi \,d\mu_{f, \phi}.
\end{eqnarray*}
	Taking the limit when $\varepsilon$ goes to zero the Brin-Katok entropy formula implies $$h_{\mu_{f, \phi}}(f)=\lim_{\varepsilon\to 0}\limsup_{n\to\infty}-\frac{1}{n}\log\mu_{f, \phi}(B_\varepsilon(x, n))=\log\lambda_{f, \phi}-\int \phi \,d\mu_{f, \phi}.$$

	%and thus 
	%$$\log\lambda_{f, \phi}=h_{\mu_{f, \phi}}(f)+\int \phi \,d\mu_{f, \phi}.$$
	
	Recalling that the topological pressure $P_{f}(\phi)$ is equal to $\log\lambda_{f, \phi}$ we have proved that $\mu_{f, \phi}$ is an equilibrium state for $(f, \phi)\in \mathcal{H}_{\sigma}.$ By the uniqueness established in \cite{RV16}  we conclude the proof of  Theorem~\ref{formalismo}.

A central limit theorem for $\mu_{f, \phi}$ can be obtained from the exponential decay of correlations, as stated in Theorem \ref{TCL}. Its proof is obtained by applying 
	a non-invertible case of an abstract central limit theorem due to Gordin, which can be found in [\cite{VianaStoch}, Theorem 2.11].  The reader can verify the steps of a similar proof in [\cite{RS17}, Theorem E].

%\begin{Teo}[Theorem \ref{TCL}]\label{tcl}
%	Let $\varphi $ be an $\alpha$-H\"older continuous function and let $\tilde{\sigma} \geq 0$ be defined by 
%	$$\tilde{\sigma}^{2}= \int \psi^2  \ d\mu_{f,\phi} + 2\displaystyle\sum_{n=1}^{\infty}  \int \psi (\psi \circ f^n) \ d\mu_{f,\phi}, \quad \mbox{where} \quad \psi = \varphi - \int \varphi \ d\mu_{f,\phi} .$$
%	Then $\tilde{\sigma}$ is finite and $\tilde{\sigma} = 0$ if and only if $\varphi = u \circ f - u $ for some $u \in L^{2}(\mu_{f,\phi})$. On the other hand, if $\tilde{\sigma} >0 $, then given any interval $A\subset \mathbb{R}$,
%	$$\mu_{f, \phi}\left(x\in M : \frac{1}{\sqrt{n}} \sum_{j=0}^{n} \left( \varphi(f^j(x))   - \int \varphi \ d\mu_{f,\phi} \right) \in A \right) \to \frac{1}{\tilde{\sigma} \sqrt{2\pi}} \int_{A} e^{-\frac{t^2}{2\tilde{\sigma}^2}} \ dt,$$ 
%	as $n$ goes to infinity.
%\end{Teo}

%%%%%%%%%%%%%%%%%%%%%%%%%%%%%%%%%%%%%%%%%%%%%%%%%%%%%%%%%%%%%%%%%%%%%%%% Analyticity  %%%%%%%%%%%%%%%%%%%%%%%%%%%%%%%%%%%%%%%
%%%%%%%%%%%%%%%%%%%%%%%%%%%%%%%%%%%%%%%%%%%%%%%%%%%%%%%%%%%%%%%%%%%  

\section{Analyticity of thermodynamical quantities with respect to the potential} \label{anal}
In this section we treat the analyticity of the thermodynamical quantities as the potential varies. Since these quantities are intrinsically related with the spectrum of the transfer operator, we will analyze its behavior under analytic perturbations. 

More specifically, the spectral gap property of the transfer operator allows us to consider the projection operator defined in Subsection~\ref{defprojection}. Using the perturbation theory due to Kato in \cite{Kato}, we prove the analyticity of this operator. Finally, we can describe the thermodynamical quantities in terms of  the projection operator and thus, the analytical dependence will follow.

Since we are fixing the underlying dynamics and varying only the potential
%  In other words, we fix the underlying dynamics and vary only the potential. Since there is no risk of confusion
 we simplify the notation by omitting the dynamics as follows $\mathcal{L}_{f,\phi}= \mathcal{L}_{\phi}$, $\mu_{f,\phi}= \mu_{\phi}$, $\nu_{f,\phi}=\nu_{\phi}$ and $h_{f,\phi}= h_{\phi}$.

We begin defining analyticity for operators on Banach spaces. Most properties of the classical analytic functions setting remain true in this context.% as we summarize below.

Let $X,Y$ be Banach vector spaces. Denote by $L^{k}_{s}(X,Y)$ the space of symmetric $k$-linear maps from the $k$-fold product $X^k \!\! := \!\!  X\times\! \cdots \! \times \!\! X$ into $Y$. Given $T_k\in L^{k}_{s}(X, Y)$ and $(H,\cdots,H) \! \in \! X^k$ we write $T_k( H^k)\!\! :=T_k(H, \cdots, H).$

We say that a mapping $\Gamma:U\subset X\to Y$ defined on an open subset $U\subset X$ is \emph{analytic} if for all $x\in U$ there exist $\varepsilon>0$ and $T_k:=T_k(x)\in L^{k}_{s}(X,Y)$, for every $k\geq 1$, depending only on $x$ such that 
$$\Gamma(x+H)=\Gamma(x)+\sum_{k=1}^{+\infty}\frac{1}{k!}T_ k(H^k)$$
for all $H$ in an $\varepsilon$-neighborhood of zero and the series is uniformly convergent. %The power series above is called  {\em Taylor series} as in the classical complex case. 

Given a potential $\phi\in C^{\alpha}(M)$ it easily follows that  $\mathcal{L}_{\phi}(\psi)\in C^{\alpha}(M)$ for any $\psi\in C^{\alpha}(M)$. In \cite{BCV} it was proved that the application which associates $\phi\in C^{\alpha}(M)$ to the transfer operator $\mathcal{L}_{\phi}:C^{\alpha}(M)\to C^{\alpha}(M)$ is analytic. 
%\begin{Prop}\label{p41}
%For any $\alpha> 0$ the map 
%$$ 
%\begin{array}{cccc}
%\Gamma \ : & \! C^{\alpha}(M)   & \! \longrightarrow
%& \! L(C^{\alpha}(M),C^{\alpha}(M))  \\
%& \! \phi & \! \longmapsto
%& \! \mathcal{L}_{\phi}
%\end{array}
%$$
%is analytic.
%\end{Prop}

In the next theorem we prove the analyticity on the potential of the projection mapping. We follow closely the ideas of Sarig~\cite[Theorem 5.6]{NSarig}.

\begin{Teo} \label{analiticidade} % Let $\{\mathcal{L}_\phi\}_{\phi\in\mathcal{P}_\sigma}$ be a family of transfer operators. %associated to potentials of $\mathcal{P}_\sigma$.
Given $\phi_0  \in\mathcal{P}_\sigma$, let $\lambda_{\phi_0}$ be the spectral radius of $\mathcal{L}_{\phi_0}$ and  let $\gamma$ be a  closed smooth simple curve which separates $\lambda_{\phi_0}$ from the rest of the spectrum. Then the projection mapping
\begin{equation*}\label{projecao}
E(\phi):=   \frac{1}{2\pi i}\int_{\gamma}(zI-\mathcal{L}_\phi)^{-1}\, dz
\end{equation*}
is analytic in a neighborhood of $\phi_0$ contained in $\mathcal{P}_\sigma$.
\end{Teo}

\begin{proof} 
Consider the set
\[
\Omega := \{(z,\phi) \in \mathbb{C}\times \mathcal{P}_\sigma: (zI - \mathcal{L}_{\phi}) \,\, \textrm{has a bounded inverse}\}.
\]	
	
Note that $\Omega$ is an open set in the product topology. In fact, given $(\overline{z}, \overline{\phi})\in \Omega$ since $\mathcal{P}_\sigma$ is open   and 
%there exists $\delta_1>0$ such that $B(\overline{\phi}, \delta_1)\subset \mathcal{P}_\sigma$. On the other hand,  Proposition \ref{p41} assures that 
the map $\phi \in \mathcal{P}_\sigma \mapsto \mathcal{L}_{\phi}$ is analytic (and therefore continuous) we can obtain   $\delta> 0$ sufficiently small such that  for all $\psi \in B(\overline{\phi}, \delta)$ its transfer operator $\mathcal{L}_{\psi}$    is  close to  $\mathcal{L}_{\overline{\phi}}$. Moreover we have 
 %$\delta_3$ sufficiently small such that $z\in B(\overline{z}, \delta_3)$, we get
  $(zI - \mathcal{L}_{\psi})$ is  close to $
(\overline{z}I - \mathcal{L}_{\overline{\phi}}) $ 
for all  $z \in B(\overline{z}, \delta)$ and for all  $\psi \in B(\overline{\phi}, \delta)$. Lemma~\ref{openproperty} assures that 	 $(zI - \mathcal{L}_{\psi})$ has a bounded inverse for all  $z \in B(\overline{z}, \delta)$ and for all  $\psi \in B(\overline{\phi}, \delta)$ and, thus,  $\Omega$ is open.

Let $\phi_0 \in\mathcal{P}_\sigma$ and let  $\gamma$  be a closed smooth  simple curve that separates  $\lambda_{\phi_0}$ and  $Spec ({\mathcal{L}_{\phi_0}})\setminus\{\lambda_{\phi_0}\}$. We can assume that $\gamma$ contains  $\lambda_{\phi_0}$ in its interior and $Spec ({\mathcal{L}_{\phi_0}})\setminus\{\lambda_{\phi_0}\}$ in its exterior. In this way, each  $z \in \gamma$  belongs to the set $Res(\mathcal{L}_{\phi_0})$ and thus $(zI - \mathcal{L}_{\phi_0})$ has a bounded inverse. Furthermore, since $\gamma$ is compact and the existence of a bounded inverse is an open property (see Lemma \ref{openproperty}), there exists $\varepsilon>0$  small enough such that $B(\phi_0, \varepsilon) \subset \mathcal{P}_\sigma$ and 
$
\mbox{for all} \ z\in \gamma  \ \mbox{and} \  \mbox{ for all} \ \phi\in B(\phi_0,\varepsilon ) \  \mbox{we have } \! (z, \phi)\in \Omega$.  In particular, $\Omega$  contains the set $\gamma \times B(\phi_0, \varepsilon)$.

The {\em resolvent map}
\[
R(z,\phi) : = (zI - \mathcal{L}_{\phi})^{-1}
\]
is well-defined and bounded on $\Omega$.
%$ z\in \gamma$ and any $\phi\in B(\phi_0,\varepsilon )$.

In order to prove that the projection    $E(\phi)=\dfrac{1}{2\pi i}\displaystyle\int_{\gamma} (zI - \mathcal{L}_{\phi})^{-1}\, dz$ is analytic,  we will show that the resolvent map is analytic on $\Omega$.
%
%Note that, to prove the veracity of the above statement, it is enough to show that the map
%\[
%(z,\phi) \mapsto R(z, \phi)=(zI - \mathcal{L}_{\phi})^{-1}
%\]
%is analytic on $\Omega$. 

%Then, we will show that $R(z, \phi)$ can be written as a convergent (in the operator norm) power series in $(z-z_0)$ and $(\phi - \phi_0)$ for all $(z_0,\phi_0)\in \Omega$. 

Given $(z_0,\phi_0)\in \Omega$, using the proof of Lemma~\ref{openproperty} we can write
\begin{eqnarray}\label{resolvente}
R(z,\phi)& =& (z_0I - \mathcal{L}_{\phi_0} - [(z_0 - z)I + (\mathcal{L}_{\phi} - \mathcal{L}_{\phi_0})] )^{-1} \nonumber  \\ 
 &=&R(z_0, \phi_0) (I - [(z_0 - z)I + (\mathcal{L}_{\phi} - \mathcal{L}_{\phi_0})]  R(z_0, \phi_0))^{-1} \nonumber \\
& = & R(z_0, \phi_0) \sum_{n=0}^{\infty} \left[ ((z_0 - z)I + (\mathcal{L}_{\phi} - \mathcal{L}_{\phi_0}))  R(z_0, \phi_0)  \right]^n,
\end{eqnarray}
provided that $\| (z_0 - z)I + (\mathcal{L}_{\phi} - \mathcal{L}_{\phi_0})\| < 1/{\|R(z_0, \phi_0)\|}$.
Note that the inequality is satisfied in a small neighborhood  $V(z_0,\phi_0)$ of $(z_0, \phi_0)$ and 
 the convergence is uniform 
on compact subsets  of that neighborhood. Moreover, 
% $\{ (z, \phi): \| (z_0 - z)I + (\mathcal{L}_{\phi} - \mathcal{L}_{\phi_0})\| < 1/{\|R(z_0, \phi_0)\|} \}$. 
since the map $\phi \in \mathcal{P}_\sigma \mapsto \mathcal{L}_{\phi}$ is analytic at $\phi_0$  we can conclude that $R(z, \phi)$ can be expressed as a double  power series  in $(z-z_0)$ and $(\phi - \phi_0)$  in $V(z_0,\phi_0)$. This ensures that the resolvent map is analytic at $(z_0, \phi_0)$. Since the choice of $(z_0, \phi_0)$ was arbitrary  the resolvent map is analytic on $\Omega$.

Now to prove that the projection is analytic we write  $R(z, \phi)$ as a power series  in $(\phi - \phi_0)$ with coefficients depending on $z$. 
We can set $z_0=z$ in equation~\eqref{resolvente} obtaining 
$$R(z,\phi)= R(z, \phi_0) \sum_{n=0}^{\infty} \left[  (\mathcal{L}_{\phi} - \mathcal{L}_{\phi_0})  R(z, \phi_0)  \right]^n .$$
 
Again,  $ \mathcal{L}_{\phi}$ is analytic  at $\phi_0$ and therefore $ (\mathcal{L}_{\phi} - \mathcal{L}_{\phi_0}) $ can be written as a power series   in $(\phi - \phi_0)$. 
% it can be expanded into a power series in $(\phi-\phi_0)$, and this power series converges in norm uniformly on some closed disc centered at $\phi_0$. Replacing this series above, expanding, and collecting terms, the result is a power series in $(\phi - \phi_0)$ and $(z-z_0)$ which converges in norm uniformly on some
%compact neighborhood $V(z_0,\phi_0)$ of $(z_0,\phi_0)$. 
If we collect the terms which multiply $(\phi - \phi_0)^n$, we obtain the following series
expansion for $R(z,\phi)$ on $V(z,\phi_0)$: 
\begin{equation*}
R(z,\phi) = R(z, \phi_0) + \sum_{n=1}^{\infty} A_n(z) (\phi - \phi_0)^n,
\end{equation*}
where $A_n(z)$ are operator valued functions which are analytics on $\{\phi: \|\phi - \phi_0\|<\eta \}$ for some $\eta = \eta(z, \phi_0)$. 

By compactness we can cover 
the set $\{(z,\phi): z\in \gamma, \|\phi - \phi_0\|\leq \eta\}$ with a finite number of neighborhoods $V(z,\phi_0)$ as above. Therefore there exists $\eta>0$ small  enough such that for any $z\in \gamma$ and $\|\phi-\phi_0\|< \eta$ we have
\begin{equation*}
R(z,\phi) = R(z, \phi_0) + \sum_{n=1}^{\infty} A_n(z) (\phi - \phi_0)^n,
\end{equation*}
where $A_n(z)$ are continuous on $\gamma$ and the series converges uniformly in
norm.

Integrating over $\gamma$, we conclude that $E(\phi)$ has a norm convergent
power series expansion as follows
\begin{eqnarray*}
E(\phi) & = & \dfrac{1}{2\pi i} \int_{\gamma} R(z, \phi)\, dz\\
& = & \dfrac{1}{2\pi i} \int_{\gamma} \left[ R(z, \phi_0) + \sum_{n=1}^{\infty} A_n(z) (\phi - \phi_0)^n\right]dz\\
&= & E(\phi_0 ) + \frac{1}{2\pi i} \sum_{n=1}^{\infty}(\phi - \phi_0)^n \int_{\gamma} A_n(z)\,dz .
\end{eqnarray*}

Then the projection is analytic in a neighborhood of $\phi_0$. 
\end{proof}

Now we prove Theorem~\ref{analiticidade potencial}. We start by showing that the topological pressure function is analytic.

Given $\phi_0\in\mathcal{P}_\sigma$ let $\lambda_{\phi_0}$ be the spectral radius of $\mathcal{L}_{\phi_0}$. Consider $\gamma$ a closed smooth simple curve which contains $\lambda_{\phi_0}$ in its interior and separates $\lambda_{\phi_0}$ from the rest of the spectrum.  Recall that equation \eqref{eq.radius} gives us bounds for the spectral radius of a transfer operator depending only on the map $f$ (which is fix here) and on the potential.  From this, if $\phi$ is close enough to $\phi_0$ we infer that the spectral radius $\lambda_{\phi}$ is also in the interior of $\gamma$.
%By Lemma~\ref{openproperty} we know that if $\varepsilon$ is sufficiently small then $\gamma$ contains the spectral radius $\lambda_\psi$ of $\mathcal{L}_\psi$ for every $\psi$ in a $\varepsilon$-neighborhood of $\phi$. 
Moreover, by  considering a smaller neighborhood of $\phi_0$, if needed, we can assume by Lemma~\ref{openproperty} that  $\gamma$ also separates $\lambda_\phi$ from the rest of the spectrum of $\mathcal{L}_\phi$. Hence, applying Corollary~\ref{eigenprojection}, we conclude that 
$$ E_{\phi}:=E(\phi) =   \frac{1}{2\pi i}\int_{\gamma}(zI-\mathcal{L}_\phi)^{-1}\, dz$$
is the eigenprojection of $\lambda_\phi$
for $\mathcal{L}_\phi.$ 

We claim that the spectral radius function is analytic on a neighborhood of $\phi_0$. Indeed, since the family $\{\mathcal{L}_\phi\}_{\phi\in\mathcal{P}_\sigma}$ has the spectral gap property (and by Corollary~\ref{eigenprojection}) it follows that dim(Im$\left(E_\phi\right))=1$ for every $\phi\in\mathcal{P}_\sigma$. In particular,  there is  $\varphi\in C^0(M)$ such that $E_{\phi_0}(\varphi)\neq 0$  and by the well-known Hahn-Banach Theorem there exists $\eta\in\left(C^0(M)\right)^{*}$ satisfying $\eta(E_{\phi_0}(\varphi))\neq 0$. By continuity of $\eta(E_{\phi_0}(\varphi))$, we have  $\eta(E_{\phi}(\varphi))\neq 0$ for every $\phi$ in a small neighborhood of $\phi_0$. 

Now  define   the   mapping  $$\Psi(\phi):= \eta(\mathcal{L}_{\phi}(E_{\phi}(\varphi)) = \int \mathcal{L}_{\phi} \circ E_{\phi}(\varphi) \, d\eta .$$ 
By the analyticity of the transfer operator and the projection map, $\Psi$  is analytic at $\phi_0$. Corollary~\ref{eigenprojection} guarantees that Im$\left(E_{\phi}\right)$ is the eigenspace associated to $\lambda_{\phi}$ then it follows
\begin{eqnarray*}
\eta(\mathcal{L}_{\phi}(E_{\phi}(\varphi))= \int \mathcal{L}_{\phi} \circ E_{\phi}(\varphi) \, d\eta = \int \lambda_{\phi} E_{\phi}(\varphi) \, d\eta =  \lambda_{\phi} \eta(E_{\phi}(\varphi)).
\end{eqnarray*}
Since $\eta(E_{\phi}(\varphi))\neq 0$ we can write   $$ \lambda_{\phi}= \dfrac{\eta(\mathcal{L}_{\phi}(E_{\phi}(\varphi)))}{ \eta(E_{\phi}(\varphi))}$$
and conclude that  the map $\phi \in \mathcal{P}_\sigma \mapsto  \lambda_{\phi} $ is analytic at $\phi_0$. Recalling that for all $\phi \in \mathcal{P}_\sigma$ the topological pressure $P_{f}(\phi)$ satisfies $\lambda_{\phi}=\exp(P_{f}(\phi))$  we have  as an immediate consequence that the map $\phi \in \mathcal{P}_\sigma \mapsto P_{f}(\phi)$ is analytic.  Thus we prove item \ref{i} of Theorem~\ref{analiticidade potencial}. 

Let $E_{\phi}^0 = \left\{ \psi \in C^{\alpha}\left(M\right) : \int \psi \ d\nu_{ \phi} = 0 \right\}$ and let $E_{\phi}^1$ be the eigenspace associated to the spectral radius $\lambda_{\phi}$ of $\mathcal{L}_{\phi}$. As in Theorem~\ref{gap}, we decompose $C^{\alpha}\left(M\right)$ as a direct sum of $E_{\phi}^0$ and $E_{\phi}^1$: given $\varphi \in C^{\alpha}\left(M\right)$ we can write 
$$\varphi = \left[ \varphi - \left( \int\varphi \ d \nu_{ \phi}\right) \cdot h_{\phi}\right] + \left[ \left( \int \varphi \ d\nu_{\phi} \right)\cdot h_{\phi} \right]= \varphi_0 + \varphi_1$$
with $\varphi_0 \in E_{\phi}^0$  and $\varphi_1 \in E_{\phi}^1$.

%In fact, since $\int h_\phi\ d\nu_\phi=1$ we have that $\varphi_0:= \left[ \varphi - \int\varphi \ d \nu_{\phi}\cdot h_{\phi}  \right]$ satisfies 
%$$\int\varphi_0\ d\nu_{ \phi}=\int \left(\varphi - \int\varphi \ d \nu_{ \phi}\cdot h_\phi\right) \ d\nu_\phi= \int\varphi\ d \nu_{ \phi} - \int\varphi \ d \nu_{ \phi}\cdot\int h_\phi\, d \nu_{ \phi} =0$$
%and so it is an element of $E_{\phi}^0$. Moreover $\varphi_1 = \left[ \left( \int \varphi \ d\nu_{\phi} \right)\cdot h_{\phi} \right]$ belongs to $E_{\phi}^1$ because
%$$\mathcal{L}_\phi(\varphi_1)=\mathcal{L}_\phi\left(\int \varphi \ d\nu_{\phi}\cdot h_{\phi}\right)=\int \varphi \ d\nu_{\phi}\cdot\mathcal{L}_\phi(h_{\phi})=\int \varphi \ d\nu_{\phi}\cdot \lambda_\phi h_\phi=\lambda_\phi\varphi_1.$$
Now, considering $\varphi\equiv \textbf{1}$, we derive  that $\varphi_1=h_{\phi}$ and this implies that $h_{\phi}$ is the projection of the function $\textbf{1}$ on the space $E_{\phi}^1.$  It follows from Corollary~\ref{eigenprojection}
$$h_{\phi}=E_{\phi}(\textbf{1})=\frac{1}{2\pi i}\int_{\gamma}(zI-\mathcal{L}_{\phi})^{-1}\, dz\, (\textbf{1}).$$

Applying Theorem~\ref{analiticidade} we conclude that $h_{\phi}$ varies analytically with respect to $\phi$ which proves item \ref{ii}. 

Moreover, since the projection of any $\varphi \in C^{\alpha}\left(M\right)$ on the space $E_{\phi}^1$ is given by $\varphi_1 = \left[ \left( \int \varphi \ d\nu_{\phi} \right)\cdot h_{\phi} \right]$ it follows 
$$E_{\phi}(\varphi)=\varphi_1=\int \varphi \ d\nu_{\phi}\cdot E_{\phi}(\textbf{1}).$$
Therefore from the relation
$$\nu_{\phi}(\cdot)=\frac{E_{\phi}(\cdot)}{E_{\phi}(\textbf{1})}$$
we obtain that the map $\phi\mapsto \nu_{\phi}\in (C^{\alpha}(M))^*$   is analytic, thus proving item \ref{iii} of Theorem~\ref{analiticidade potencial}. In particular, item \ref{iv}  %fthe analyticity of the equilibrium state $\mu_{\phi}$ 
follows from the previous results since that $\mu_{\phi}=h_{\phi}\nu_{\phi}.$

%%%%%%%%%%%%%%%%%%%%%%%%%%%%%Applications%%%%%%%%%%%%%%%%%%%%%%%%%%%
%%%%%%%%%%%%%%%%%%%%%%%%%%%%%%%%%%%%%%%%%%%%%%%%%%%%%%%%%%%%%%%%%%%%

	\section{Applications: a class of non-uniformly hyperbolic skew products}\label{skew product} 
	
	In this section we extend our results to a wider class by considering a family of skew products over non-uniformly expanding maps.
	
	Consider $N$ a compact metric space with distance $d$ and let $g:M\times N\rightarrow N$ be a continuous map uniformly contractive on $N$, i.e., there exists  $0<\lambda<1$ such that for all $x\in M$  and all $y_1,y_2\in N$   we have
	\begin{equation}\label{g}
	d\big(g(x, y_1), g(x,y_2)\big)\leq\lambda d(y_1,y_2).
	\end{equation}  
	Suppose that there exists some $\bar{y}\in N$ such that $g(x,\bar{y})=\bar{y}$ for every $x\in M.$ 
	We consider a family  $\mathcal S$ of skew-product maps $F:M \times N \to M \times N $ such that
	$$   F(x,y)= (f(x),g(x,y))  $$
	for all $(x,y)\in M\times N$, where the {\em base dynamics} $f:M\rightarrow M$ belongs to $\mathcal{F}$ and  is strongly topologically mixing and the fiber dynamics $g:M\times N\rightarrow N$ satisfies~\eqref{g}. %In order to avoid confusion we denote  base dynamics of $F$ by  $b_F$.
	
	\begin{remark}
		If we can write $N=N_1\cup \cdots \cup N_n$, where  $N_1, \dots, N_n$ are pairwise disjoint compact sets then the condition above, $g(x, \bar{y})=\bar{y}$  for all $x\in M$, can be replaced by $g_i(x, y_i)=y_i$ for all $x\in M$ and some $y_i\in N_i$, $i=1, \dots, n$.
		That  is because  we can define $n$ fiber dynamics $g_i:M\times N_i\to N_i$ by $g_i(x, y)=g(x, y)$ for $y\in N_i$, for each $i=1,\dots, n$, since  $M\times N$ is a product space and $N=N_1\cup \cdots \cup N_n$. See Example~\ref{ferradura}.
	\end{remark}
	
		Given $\sigma\in(0,1)$ we say that a continuous potential $\Phi: M\times N \to \mathbb{R}$ is a {\em $\sigma$-hyperbolic potential for $F\in\mathcal S$} if the  topological pressure of the system $(F, \Phi)$ is equal to the relative pressure on the set~$\Sigma_{\sigma}(f)\times N$. More precisely, 
$$P_F(\phi, [\Sigma_{\sigma}(f)]^c\times N)<P_F(\phi,\Sigma_{\sigma}(f)\times N)=P_F(\phi).$$
	
	We consider the family 
	$\mathcal G_\sigma$  of pairs $(F,\Phi)\in\mathcal S\times C^\alpha(M\times N)$ such that \eqref{h1} holds for $f$ and  $\Phi$ is hyperbolic  for  $F$ satisfying condition~\eqref{2 estrelas}.
	The uniqueness of  the equilibrium state for $(F, \Phi) \in \mathcal G_\sigma$ was proven in \cite{ARS}. 
	
	Given  $\Phi:M\times N\to \mathbb{R}$ a H\"older continuous potential, hyperbolic for $F$ satisfying condition~\eqref{2 estrelas}. It was proved  in [\cite{ARS}, Section 5] that   $\Phi$ induces a H\"older continuous potential $\phi:M\to \mathbb{R}$ which is hyperbolic for the base dynamics $f$ and satisfies $ \Var(\phi)\leq \Var(\Phi)$.  Therefore,  the induced potential $\phi$ satisfies condition~\eqref{2 estrelas} as well.
	
	Moreover, by [\cite{ARS}, Lemma 5.3], the unique equilibrium state $\mu_{F, \Phi}$ associated to the system $(F, \Phi)$ is given by the push-forward $\pi_{\ast}\mu_{F, \Phi}=\mu_{f, \phi}$ where $\mu_{f, \phi}$ is the unique equilibrium state of $(f, \phi)$. In other words, for every Borel set $A$ of $M\times N$ we have 
	$$\mu_{F, \Phi}(A)=\mu_{f,\phi}(\pi(A)).$$
	
	Note that the map $\pi$ is analytic and does not depend on the potential~$\Phi$. From this and item \ref{iv} of Theorem~\ref{analiticidade potencial}, we can derive the analyticity of the equilibrium state  when we fix the skew product $F$ and vary the potential $\Phi$ within the family $\mathcal{G}_{\sigma}$. 
	
	\begin{Cor}\label{analiticidade F}
		%Let $F\in\mathcal S$ and  let ${\phi}: M\times N \to \mathbb{R}$ be a H\"older continuous and $c$-hyperbolic potential for $F$. 
		The equilibrium state varies analytically on the potential within the family $\mathcal G_\sigma$.
	\end{Cor}

	In what follows, we state some statistical properties for the unique equilibrium state of $(F, \Phi) \in \mathcal G_\sigma$.

	The key idea to obtain the exponential decay of correlations for the equilibrium state associated to $(F, \Phi)$ (see the statement of the next result) is to disintegrate this measure as a product of conditional measures on stable fibers by the equilibrium state of the base dynamics. Details of a similar proof of this result can be analyzed in \cite{RS17}. Due to the similarity,  we choose to omit such proof here.

	\begin{Cor}\label{decaimento F}
		The equilibrium state $\mu_{F,\Phi}$ has exponential decay of correlations for H\"older continuous observables: there exists $0< \tau <1$ such that for every  $\varphi \in L^{1}(\mu_{F, \Phi})$ and $\psi \in C^{\alpha}(M\times N)$ there exists $K(\varphi, \psi)>0$ so that
		$$\left| \int \left(  \varphi \circ F^{n} \right)\psi\, d\mu_{F, \Phi} - \int \varphi\, d\mu_{F, \Phi} \int \psi\, d \mu_{F, \Phi}\right| \leq K(\varphi, \psi)\tau^{n}\, ,  \quad \forall n\geq 1. $$
	\end{Cor}

		We also obtain a central limit theorem for the equilibrium state $\mu_{F, \Phi}$ of the skew-product $F$ with respect to a potential $\Phi$ as considered above.
	
	\begin{Cor}\label{TCL F}
		Let $\varphi $ be an $\alpha$-H\"older continuous function and let $\tilde{\sigma} \geq 0$ be defined by 
		$$\tilde{\sigma}^{2}= \int \psi^2  \ d\mu_{F, \Phi} + 2\displaystyle\sum_{n=1}^{\infty}  \int \psi (\psi \circ F^n) \ d\mu_{F, \Phi} \quad \mbox{where} \quad \psi = \varphi - \int \varphi \ d\mu_{F, \Phi} .$$
		Then $\tilde{\sigma}$ is finite and $\tilde{\sigma} = 0$ if and only if $\varphi = u \circ F - u $ for some $u \in L^{2}(\mu_\Phi)$. On the other hand, if $\tilde{\sigma} >0 $ then given any interval $A\subset \mathbb{R}$, 
		\[ \mu_{F, \Phi}\! \left(\! x\! \in \!\! M \!\times \!N: \frac{1}{\sqrt{n}} \sum_{j=0}^{n} \left(\!\varphi(F^j(x))\! - \! \int \! \varphi \ d\mu_{F, \Phi} \right) \! \! \in A \! \right)\! \to \frac{1}{\tilde{\sigma} \sqrt{2\pi}}\! \int_{A}\! e^{-\frac{t^2}{2\tilde{\sigma}^2}}  dt,\] 
		as $n$ goes to infinity. 
	\end{Cor}

	The proof of Corollary~\ref{TCL F} also follows from Gordin's Theorem, which can be applied once we have established the exponential decay of correlations. Therefore, in view of Corollary \ref{decaimento F}, the proof of Corollary~\ref{TCL F} is exactly analogous to the one of Theorem~\ref{TCL}.

%%%%%%%%%%%%%%%%Examples%%%%%%%%%%%%%%%%%%%%%%%%%%%%%%%%%%%%%%%%

\section{Examples}\label{exemplos}

In this section we describe some examples of systems which satisfy our results. We start by presenting a robust class of non-uniformly expanding maps introduced by Alves, Bonatti and Viana \cite{ABV} and studied by Arbieto, Matheus and Oliveira \cite{AMO}, Oliveira and Viana \cite{OV08}, Varandas and Viana \cite{VV}. 

\begin{Ex} \normalfont{Consider $f:M\rightarrow M$ a $C^{1}$ local diffeomorphism defined on a compact manifold $M$. For  $\delta>0$ small and $\sigma<1$, consider  a covering $\mathcal Q=\left\{Q_{1},\dots, Q_{q}, Q_{q+1},\dots, Q_{s}\right\}$ of $M$ by domains of injectivity for $f$ and a  region $\textsl{A}\subset M$ satisfying:
		\begin{enumerate}
			\item[(H1)]   $\|Df^{-1}(x)\|\leq 1+\delta$, for every $x\in\textsl{A}$;
			\item[(H2)] $\|Df^{-1}(x)\|\leq \sigma $, for every $x\in M\setminus\textsl{A}$;
			\item[(H3)]   $A$   can be covered by $q$ elements of the partition $\mathcal Q$ with $q<\deg(f).$
		\end{enumerate}
		The authors aforementioned showed that there exists a constant $\sigma\in(0,1)$ and a set $H\subset M$ such that for every $x\in H$ we have
		\begin{eqnarray*} 
			\limsup_{n\rightarrow+\infty}\frac{1}{n}\sum_{i=0}^{n-1}\log\|Df(f^{j}(x))^{-1}\|\leq-\sigma .
		\end{eqnarray*}
		Furthermore, it was proved that  for a H\"older continuous potential $\phi:M\rightarrow\mathbb{R}$ with \emph{small variation}, i.e. $$\sup\phi - \inf\phi < \log\deg(f)-\log q,$$ it follows that the relative pressure $P(\phi, H)$ satisfies 
		$$P_{f}(\phi, H^{c})<P_{f}(\phi, H)=P_{f}(\phi).$$
		and thus $\phi$ is $\sigma$-hyperbolic for $f$. 
		
		Denote by $\mathcal{F}$ the class of $C^1$ local diffeomorphisms satisfying the conditions (H1)-(H3) and assume that every $f\in \mathcal{F}$ is strongly topologically mixing. Let $\mathcal{H}$ be the family of pairs $(f,\phi)$, such that $f \in \mathcal{F} $  and $\phi:M\to\mathbb R$  is H\"older continuous,  with small variation  and satisfies condition~\eqref{2 estrelas}. Our results imply that  in this family the equilibrium state has exponential decay of correlations, satisfies a central limit theorem and varies analytically with the potential.}
\end{Ex}

In the next example we present a family of intermittent maps, although it is a particular case of the previous example, we present it here because %of its  own interest and because
	the   reader can  check the hypotheses for the results in this paper more easily.  Moreover,  we want to compare our results for the equilibrium state with some of the ones existing in the literature also for the absolutely continuous invariant probability measure.% which in this case coincides with the SRB measure but do not coincides with the equilibrium measures considered here. 

\begin{Ex}\normalfont{Let $\beta \in(0,1)$ be a positive constant  and consider  the local diffeomorphism defined in $S^{1}$ by
		%$$
		%		f_{\beta}(x)=\left\{
		%		\begin{array}{cc}
		%		x(1+2^{\beta}x^{\beta}),&\; \textit{\normalfont{if}}\;\; 0\leq x\leq \frac{1}{2}\\\\
		%		x-2^{\beta}(1-x)^{1+\beta},&\; \textit{\normalfont{if}}\;\; \frac{1}{2}\leq x\leq 1.
		%		\end{array}
		%		\right.
		%		$$
		
		$$
		f_{\beta}(x)=\left\{
		\begin{array}{cc}
		x(1+2^{\beta}x^{\beta}),&\; \textit{\normalfont{if}}\;\; 0\leq x\leq \frac{1}{2}\\\\
		2x-1,&\; \textit{\normalfont{if}}\;\; \frac{1}{2}\leq x\leq 1.
		\end{array}
		\right.
		$$
		%Then $f_{\alpha}$ is a non-uniformly expanding map. Furthermore, s
		It is straightforward to check that $f_{\beta}$ is strongly topologically mixing. Consider $\mathcal{F}$ the class of $C^1$ local diffeomorphisms $\{f_{\beta}\}_{\beta\in(0,1)}$ and let $\phi:S^1\to\mathbb R$ be H\"older continuous with small variation and satisfying condition~\eqref{2 estrelas}.

		Then the unique equilibrium state satisfy a central limit theorem, has exponential decay of correlations, and varies analytically with the potential.

  We remark that the statistical properties mentioned above were previously obtained in \cite{riveraCMP}, along with the real analyticity of the topological pressure. We note that  the meaning of analyticity in this paper differs from the one in the aforementioned work.  
	
	Now, still considering the intermittent maps, we discuss a particular family of potentials of intrinsic interest in the literature. 
		
		For each $\beta\in (0, 1)$ consider $\{\phi_{\beta, t}\}_t$ the family of potentials defined by $\phi_{\beta, t}=-t\log|f'_\beta|$.  %Indeed,
 The variation of $\phi_{\beta, t}$ can be made small if $t$ is small:		\begin{eqnarray*}
			\left|\phi_{\beta, t}(x)-\phi_{\beta, t}(y)\right|=\left|-t\log|f'_\beta|(x)+t\log|f'_\beta|(y)|\right|	&=&|t|\log\frac{|f'(y)|}{|f'(x)}\\
			&\leq &|t|\log(2+\beta). 
		\end{eqnarray*}
Then for $|t|\leq t_0$ sufficiently small this family satisfies the hypotheses of our results.
%Thus we can apply our results to conclude that for each system $(f_\beta, \phi_{\beta, t})$ there exists a unique equilibrium state. It has exponential decay of correlations and satisfies a central limit theorem.
In particular, this implies that for $|t|\leq t_0$ the topological pressure function $t\to P_{f_\beta}(-t\log|f'_\beta|)$ is analytic as well as the equilibrium state.
		
		We point out that our results can be applied for $|t|$ small, the scenario outside this range can be quite different. For example when $t=1$, it is well-known that $\phi_\beta=-\log|f'_{\beta}|$ admits two ergodic equilibrium states: the Dirac measure centered at the fixed point and the unique invariant probability measure
	 absolutely continuous with respect to Lebesgue (a.c.i.p.).
 
For the a.c.i.p. the correlations for  H\"older observables decay polynomially [\cite{LSV},\cite{Young99}]
	 %This model also admits a unique  absolutely continuous invariant probability measure \cite{LVS}  for which the correlations of H\"older observables decays polynomially [\cite{LSV}, \cite{Young}]. 
		The  polynomial decay is optimal  \cite{Sarig2002}.	 Also, a central limit theorem holds  for $\beta<1/2$ \cite{Sarig2002}  and  holds for $\beta >1/2$ considering  observables vanishing in a neighborhood of the fixed point and with zero integral \cite{Gouezel2002}. Moreover, in \cite{Baladi-Todd} and in \cite{Korepanov2016}, independently, it was proven that  the a.c.i.p.  is differentiable.

		To end our discussion about this model we highlight that considering a one parameter family of potentials whose derivative near to zero behaves as a {\it polynomial}, Sarig \cite{sarig-pressao} has proven a phase transition for its topological pressure. In other words, there exists a  critical parameter for which the topological pressure is not analytic.	 %We point out that this class of potentials does not satisfy the  hypothesis of our results. 
}

\end{Ex}

%%%%%%%%%%%%%%%%%%%%%%%%%%%%%%%
%%%%%%%%%%%%%%%%%%%%%%%%%%%%%%%%%%%%%
%%%%%%%%%% FERRADURA   %%%%%%%%%%%%%%
%%%%%%%%%%%%%%%%%%%%%%%%%%%%%%%
%%%%%%%%%%%%%%%%%%%%%%%%%%%%%%%%%%%%%

As an application of our corollaries we describe a family of partially hyperbolic horseshoes. This class of maps was defined by D\'iaz, Horita, Rios and Sambarino in \cite{diazetal} and has been studied in several works [see \cite {LOR}, \cite{RV16} and \cite{RS15}]. 
In particular, in \cite{RS15}, it was shown that this family can be modeled by a skew product which base dynamics is strongly topologically mixing and non-uniformly expanding.

\begin{Ex} \label{ferradura} \normalfont{Consider the cube $R= [0,1]\times[0,1]\times[0,1]\subset\mathbb{R}^3$ and  the parallelepipeds
		$$R_0 =[0,1]\times [0,1]\times [0,1/6]\quad \mbox{and} \quad R _1=[0,1]\times [0,1]\times [5/6,1].$$ 
		For $(x,y,z)\in R_0$ consider a map defined as
		$$ F_{0}(x,y,z) =(\rho x , f(y),\beta z),$$
		where $0 < \rho <{1/3}$, $\beta> 6$ and  $$f(y) =\frac {1}{1 - \left(1-\frac{1}{y}\right)e^{-1}}.$$
		For $(x,y,z)\in R_1$ consider a map defined  as
		$$F_{1}(x,y,z)  = \left(\frac{3}{4}- \rho x , \eta(1 - y) ,\beta_{1} \left(z - \frac{5}{6} \right)\right),$$
		where  $0<\eta< {1/3}$ and $3< \beta_1 < 4$.
		We define the horseshoe map $F$ on $R$ as
		$$F\vert_{R_0}=F_0,\quad F\vert_{R_1}=F_1,
		$$
		with  $  R\setminus(R_0\cup R_1)$ being mapped injectively outside $R$.
		
		For fixed parameters  $\rho, \beta, \beta_1$ and $\eta$ satisfying conditions above, the non-wandering set of $F$ is  partially hyperbolic, see  \cite{diazetal}.
		Let $\Omega$ be the maximal invariant set for $F^{-1}$ on the cube $R$ and consider $\Phi:R_0\cup R_1\to \mathbb{R}$ a H\"older continuous potential with \emph{small variation}, i.e.		\begin{equation}\label{potencial ferradura}
		\sup\Phi - \inf\Phi < \frac{\log\omega}{2}, \quad \mbox{where}\quad \displaystyle\omega=\frac{1+\sqrt{5}}{2}.
		\end{equation}
		The partially hyperbolic horseshoe admits only one equilibrium state associated to the potential $\Phi$ with small variation [see \cite{RS17}]. Moreover, in \cite{RS15} it was shown that small variation implies that the potential is hyperbolic.  }
	Also in  \cite{RS15}, it was proved that the  map $F^{-1}$  can be written as a skew product whose base dynamics, which they call projected map, is  strongly topologically mixing and non-uniformly expanding. Moreover the fiber dynamics is a uniform contraction. 
	
	Let  $\mathcal{S}$ be the family of partially hyperbolic horseshoes $F$, depending on the parameters $\rho, \beta, \beta_1$ and $\eta$ as above. 
	Consider $\mathcal{G}$ the family of pairs $(F^{-1}, \Phi)$ such that $F\in  \mathcal{S}$ and $\Phi$ satisfies \eqref{potencial ferradura} and condition~\eqref{2 estrelas}. 
	Note that  each pair $(F^{-1}, \Phi) \in \mathcal{G}$ satisfies the hypotheses of our results. Since the equilibrium state, $\mu_{F^{-1}, \Phi}$, associated to $(F^{-1}, \Phi)$ coincides with the equilibrium state, $\mu_{F, \Phi}$, associated to  $(F, \Phi)$  we can  apply our results to  $\mu_{F, \Phi}$.
	In other words,  the correlations of  the equilibrium state  $\mu_{F, \Phi}$  decay exponentially, a central limit theorem holds and $\mu_{F, \Phi}$  varies analytically with respect to the potential.
We note that the statistical properties were previously obtained in \cite{RS17}, the novelty is the analyticity. 
%{\color{purple}We stress that the two first results mentioned above were already obtained in \cite{RV16}, the main novelty  for this example is the analyticity of the equilibrium state with respect to the potential. }
\end{Ex}

\end{document}